\documentclass[reqno]{amsart}

\usepackage{lmodern}

\usepackage{float}
\usepackage[colorlinks=true,
            citecolor=blue,
            linkcolor=black,
            urlcolor=blue,
            linktocpage=true,
            hypertexnames=false]{hyperref}

\usepackage{caption}
\usepackage{ragged2e}
\usepackage{graphicx}
\usepackage{tikz}
\usetikzlibrary{patterns, decorations.markings}

\usepackage{enumerate}

\usepackage{amsthm}

\theoremstyle{plain}
\newtheorem{theorem}{Theorem}[section]
\renewcommand{\thetheorem}{%
	\ifnum\value{section}>0
	\thesection.%
	\else
	\thechapter.%
	\fi
	\arabic{theorem}%
}
\newtheorem{lemma}[theorem]{Lemma}
\newtheorem{proposition}[theorem]{Proposition}
\newtheorem{corollary}[theorem]{Corollary}

\newtheorem{remark}[theorem]{Remark}

\newtheorem*{remark*}{Remark}
\newtheorem*{question*}{Open question}
\newtheorem*{theorem*}{Theorem}
\newtheorem*{lemma*}{Lemma}
\newtheorem*{proposition*}{Proposition}
\newtheorem*{corollary*}{Corollary}
\newtheorem*{propriety*}{Property}
\newtheorem*{definition*}{Definition}
\newtheorem*{example*}{Example}
\newtheorem*{notation*}{Notation}

\usepackage{amsmath, amsfonts, amssymb, dsfont}

\newcommand{\A}{\mathcal{A}}

\newcommand{\E}{\mathbb{E}}

\renewcommand{\L}{\mathbb{L}}

\newcommand{\N}{\mathbb{N}}

\renewcommand{\P}{\mathbb{P}}

\newcommand{\R}{\mathbb{R}}

\newcommand{\Z}{\mathbb{Z}}

\renewcommand{\sin}[1]{\operatorname{sin}\left(#1\right)}
\renewcommand{\exp}[1]{\operatorname{exp}\left(#1\right)}

\renewcommand{\sin}[1]{\operatorname{sin}\left(#1\right)}
\renewcommand{\exp}[1]{\operatorname{exp}\left(#1\right)}

\newcommand{\Vol}[1]{\mathcal{L}_{d}\left(#1\right)}

\newcommand{\card}[1]{\left|\left\{#1\right\}\right|}

\makeatletter
\def\esp{\@ifnextchar[{\@withe}{\@withoute}}
\def\@withe[#1]#2{\E_{#1}\left[#2\right]}
\def\@withoute#1{\E\left[#1\right]}
\def\var{\@ifnextchar[{\@withv}{\@withoutv}}
\def\@withv[#1]#2{\mathbb{V}_{#1}\left(#2\right)}
\def\@withoutv#1{\operatorname{Var}\left(#1\right)}
\makeatother
\newcommand{\cov}[2]{\operatorname{cov}\left(#1, #2\right)}

\renewcommand{\sin}[1]{\operatorname{sin}\left(#1\right)}

\renewcommand{\exp}[1]{\operatorname{exp}\left(#1\right)}

\newcommand{\Id}{\operatorname{Id}}
\newcommand\restr[2]{{
		\left.\kern-\nulldelimiterspace 
		#1
		\littletaller
		\right|_{#2}
}}

\numberwithin{equation}{section}

\title[Stationary increment perturbation of point processes]{Regularization of a stationary point process by a stationary increments perturbation}
\author{Loïc Thomassey, Raphaël Lachièze-Rey, Assaf Shapira}

\begin{document}
	
	\begin{abstract}
		We present a novel procedure based on the result of \cite{Thomassey2025} where a stationary point process is regularized through the convolution with a continuous random field with stationary increments, in the sense that the dependency  between distant points is weakened, and the potential peaks in the spectrum (or Bragg peaks), reminiscent of a periodic behavior, are erased. We use this procedure to efficiently generate a hyperuniform point process in dimension $1$ using a fractional Brownian Motion, simulating $n$ points with complexity $n \log(n)$.
	\end{abstract}
	
	\maketitle
	
	
	\section{Introduction and motivations}
	
	\subsection{Perturbations of the lattice}\label{sec:introduction}
	
	A point process $\{\xi_{n} : n \in \Z\}$ is a locally finite random configuration of points in $\R^{d}$. We are interested in those configurations which are stationary, that is invariant with respect to translations, \textit{i.e.} 
	\begin{equation*}
		\xi \stackrel{(d)}{=}  \tau _{ t}\xi
	\end{equation*}
	for all $t\in \R^{d}$ where
	\begin{equation*}
		\tau_{t}\xi = \{\xi_{n}-t : n \in \Z\}.
	\end{equation*}
	An inexpensive and common procedure to generate  regular non-trivial stationary point processes is to start with the shifted lattice $\xi=\Z^{d}+U$ where $U$ is uniform on the unit square, and perturb each atom of the configuration with an independent $\R^{d}$-valued $ \mathbb{Z} ^{ d}$-stationary stochastic process $X=(X_{n})_{n\in \Z^{d}}$. This yields a perturbed point process
	\begin{equation*}
		\xi_{X}=\{n+U+  X_{ {n}} : n \in \Z\},
	\end{equation*}
	which is $ \mathbb{R}^{ d}$-stationary and inherits some of the spectral properties of the initial lattice. By the {spectral measure} of a stationary point process $ \xi $, we refer here to the Bartlett's spectrum of the process (see \cite{Daley2008}), that is the unique positive and locally finite measure $\mathcal S_{ \xi  }$ satisfying 
	\begin{equation}\label{eq:phase_equation}
		\var{\int_{\R^{d}} \varphi(t) \, \xi (\mathrm{d}t)} = \int_{\R^{d}}|\hat{\varphi}(t)|^{2}\,\mathcal{S}_{\xi}(\mathrm{d}t)
	\end{equation}
	for any smooth compactly supported function $f$. Here 
	\begin{equation*}
		\hat{\varphi}(t)=\int_{\R^{d}}e^{-i (x,t)}\varphi(x)\, \mathrm{d}x
	\end{equation*}
	and $\xi$ is identified with the discrete measure $\xi=\sum_{n\in \Z}\delta_{\xi_{n}}$.\medbreak 
    
    Deriving an explicit formula for the Bartlett's spectrum of a lattice perturbed by a dependent process can be a challenging task, but explicit formulas exist in a few particular cases. Thus, let us briefly assume that $(X_{n})_{n\in \Z^{d}}$ is a family of independent and identically distributed Gaussian perturbations. In that case, a quick computation shows that $\mathcal{S}_{\xi_{X}}$ is a mixture of a purely atomic measure and an absolutely continuous one as
	\begin{equation}\label{eq:bartletts_spectrum_perturbed_lattice}
		\mathcal S_{ \xi_{X}}(\mathrm{d}t)  =   (1- | \kappa(t)| ^{ 2})\mathrm{d}t  + | \kappa (t) | ^{ 2}  \mathcal S_{ \xi }(\mathrm{d}t),
	\end{equation}
	where $\kappa(t)=\esp{e^{i(t,X_{0})}}$ and
    \begin{equation*}
		\mathcal{S}_{\xi} = \sum_{n\in \Z\backslash\{0\}}\delta_{2\pi n}
	\end{equation*}
    is the Bartlett's spectrum of the shifted lattice $\xi=\Z+U$, see \cite{Coste2021}. In the preceding expression, the underlying lattice structure is apparent through the atomic part $| \kappa (t) | ^{ 2}  \mathcal S_{ \xi }(\mathrm{d}t)$.\medbreak
	
	Evaluating \eqref{eq:bartletts_spectrum_perturbed_lattice} at $   \mathsf B_{\varepsilon}=\{x \in \R^{d} : \|x\|\leq \varepsilon\}$  \   and letting $\varepsilon\to 0$ yield
	\begin{equation}\label{eq:hyperuniformity spectral}
		\lim_{\varepsilon\to 0}\varepsilon ^{ -d}\mathcal{S}_{\xi_{X}}(   \mathsf B_{\varepsilon}) = 0,
	\end{equation}
	or equivalently,  see \cite{Bjoerklund2023},
	\begin{equation}\label{eq:hyperuniformity}
		\lim_{r\to+\infty}\frac{\var{\xi_{X}(   \mathsf B_{r})}}{\Vol{   \mathsf B_{r}}} = 0,
	\end{equation}
    where $\Vol{\cdot}$ is the Lebesgue measure on $\R^{d}$.\medbreak
	
	For a Poisson process, the latter limit would be strictly positive, hence \eqref{eq:hyperuniformity} describes a cancellation phenomenon, which is known as  {\it hyperuniformity}. The study of hyperuniform point processes was initially pioneered by Torquato and al. in their seminal paper \cite{Torquato2003}, and it has gained traction as it was discovered that such processes arise naturally in many fields, ranging from physics, biology or mathematics, see \cite{Torquato2018} for a more detailed introduction or \cite{LachiezeRey2025a} for a recent mathematical survey. As such, perturbed lattices provide a simple procedure to generate non-trivial hyperuniform point processes.\medbreak
	
	In this article, we propose an alternative way of constructing non-trivial hyperuniform point processes from lattices while erasing the underlying lattice structure. The construction is  similar in spirit to the model of stationary perturbed lattice, except that it involves acting on the underlying Palm distribution instead of the shifted lattice. To avoid delving into too many technicalities in this introduction, we shall simply describe a Palm distribution as a point process $\hat{\xi}$ whose law is the distribution of a stationary point process $\xi$ conditioned on the event $(0\in \xi)$. From a theoretical point of view, Palm distributions completely characterize the law of the underlying stationary point process, since the application $\xi\longrightarrow\hat{\xi}$ is one-to-one. \medbreak
	
	The Palm distribution of the shifted lattice is the lattice itself, i.e. $\hat{\xi}=\Z^{d}$, and one may further shift all its points by a stochastic process with stationary increments $(B_{t})_{t\in \R^{d}}$, that is a $\R^{d}$-valued stochastic process satisfying 
	\begin{enumerate}
		\item $B_{0}=0$, 
		\item $B$ has stationary increments, i.e. for any $t\in \R^{d}$,
		\begin{equation*}
			\theta_{t}B \stackrel{(d)}{=} (B_{x+t}-B_{t})_{x\in \R^{d}}.
		\end{equation*}
	\end{enumerate} 	
		
	This results in a perturbed point process
	\begin{equation*}
			\hat{\xi}_{B}=\{n+B_{n} : n \in \Z^d\}
	\end{equation*}
    One of the findings of this article is the non-trivial fact that $ \hat \xi _{ B}$ is the Palm measure of a  stationary point process. Moreover, the preceding procedure can be generalized to Palm distributions others than the lattice $\Z^{d}$ in the sense that, if $\hat{\xi}$ is the Palm distribution of a stationary point process $\xi$, then 
	\begin{equation*}
		\hat{\xi}_{B}=\{x + B_{x} : x\in \hat{\xi}\}
	\end{equation*} 
	is the Palm distribution of a stationary point process, see Theorem \ref{thm:fBm peturbed lattice} below.\medbreak

	Yet, similarly to the stationary case, deriving an explicit formula for the Bartlett's spectrum of $\hat{\xi}_{B}$ is a difficult task and we shall therefore concentrate on a particular class of perturbations, known as fractional Brownian motions and their generalization to higher dimensions, fractional Brownian fields (fBf), see \cite{Mandelbrot1968}. By $d$-dimensional fBf, we mean the unique centered $\R^{d}$-valued Gaussian process $B=(B_{1}, \dots, B_{d})$ with stationary increments, independent coordinates and variogram
	\begin{equation}\label{eq:covariance_d_fbf}
		\Sigma_{t}= \cov{B_{t}^{(i)}}{B_{t}^{(j)}}_{1 \leq i,j \leq n}=\begin{pmatrix}
			\|t\|^{2h_{1}}&&(0)\\
			&\ddots&\\
			(0)&&\|t\|^{2h_{n}}
		\end{pmatrix}
	\end{equation}
	where $(h_{1}, \dots, h_{d}) \in (0,1)^{d}$ is known as the  {\it Hurst index} of $B$.\medbreak

	Finally, when $d=1$, we will say that $B$ is a fractional Brownian motion (fBm). For $h=1/2$, we recover a process with independent increments whose law coincides with the traditional distribution of a Brownian motion.
	
	\begin{theorem}\label{thm:fBm peturbed lattice}
		Let $\xi$ be a stationary point process with intensity $1$ and finite second moment, \textit{i.e.}
		\begin{equation*}
			\esp{\xi(K)^{2}}<+\infty, \quad K \text{ compact }, 
		\end{equation*}
		and $\hat{\xi}$ be its Palm distribution. If $B$ is a $d$-fBf independent of $\hat{\xi}$, then 
		\begin{enumerate}[(i)]
			\item $\hat{\xi}_{B} = \{x + B_{x} : x \in \hat{\xi}\}$ is the Palm distribution of an ergodic point process $\xi_{B}$ with intensity $1$,
			\item the Bartlett's spectrum of $\xi_{B}$ is absolutely continuous with respect to Lebesgue measure and its density (also known as \emph{structure factor}) is given by
			\begin{equation}\label{eq:structure_factor}
				s_{\xi_{B}}(t) = \esp{\int_{\R^{d}}e^{-\frac{1}{2}(\Sigma_{x}t,t)}e^{-i( t,x)} \, \hat{\xi}(\mathrm{d}x)}, \quad t\neq 0,
			\end{equation} 
			which converges absolutely on all compact subsets of $\R^{d}\backslash\{0\}$. Here, $\hat{\xi}$ is identified with a discrete measure.
		\end{enumerate}
	\end{theorem}
	
	\begin{remark}

		Theorem \ref{thm:fBm peturbed lattice} is stated for a $d$-fBf. However, the conclusion of part \textit{(i)} continues to hold if $B$ is replaced by an ergodic Gaussian process with stationary increments. The latter assumption is equivalent to requiring that the Levy measure of $B$ is non-atomic (see Section \ref{sec:fbm}), a condition that is typically straightforward to verify in practice.\medbreak

		In contrast, the derivation of the structure factor and the cancellation of the atomic component of Bartlett’s spectrum rely crucially on the specific fluctuation properties of the $d$-fBf. Such behavior need not persist when $B$ is a general Gaussian process with stationary increments.\medbreak

		Therefore, determining a general expression for the structure factor of $\xi_B$ when the variogram of $B$ exhibits slow growth remains an open problem.
	\end{remark}

	Returning to the example of the lattice $\hat{\xi}=\Z^{d}$, $\xi_{B}$ is in that case a stationary point process with structure factor
	\begin{equation*}
		\mathcal{S}_{\xi_{B}}(t) = \sum_{n \in \Z^{d}}e^{-\frac{1}{2}(\Sigma_{n}t,t)} e^{-i(t,n)}.
	\end{equation*}
	If we compare this result with the one available for a  lattice perturbed by a stationary field, see equation \eqref{eq:bartletts_spectrum_perturbed_lattice}, one may point out a notable difference. Unlike $\mathcal{S}_{\xi_{X}}$, $\mathcal{S}_{\xi_{B}}$ no longer has any atomic component reminiscent of the lattice structure of $\Z^{d}$. The latter has been completely erased by the fluctuations of the fBf. We anticipate that the elimination of the periodic order goes beyond order $ 2$, and conjecture the resulting process to be mixing, but this is outside the scope of this article, see Section \ref{sec:mixing}. \medbreak
	
	A more delicate question is related to the hyperuniformity of $\xi_{B}$. As it has absolutely continuous Bartlett's spectrum, hyperuniformity occurs if and only if
	\begin{equation}\label{eq:hyperunifomrmity_structure_factor}
		\lim_{\|t\|\to 0}s_{\xi_{B}}(t)=0.
	\end{equation}
	Establishing the limiting behavior of $s_{\xi_{B}}$ is not simple  as it requires evaluating the asymptotic behavior of an oscillatory integral. In dimension $1$, these computations are difficult but one can obtain an asymptotic for the lattice :
	
	\begin{proposition}\label{prop:degree_hyperuniformity_palm_perturbed_lattice}
		Let $\hat{\xi}=\Z$ and $B$ a fBm with index $0<h<1$, \textit{i.e.} $\var{B_t}=|t|^{2h}$. Then, 
		\begin{equation}\label{eq:asymptotic_structure_factor}
			s_{\xi_{B}}(t) \underset{t\to 0}{\sim} \alpha_{h}|t|^{1-2h}
		\end{equation}
		with
		\begin{equation}\label{eq:asymptotic_structure_factor_constant}
			\alpha_{h}=2h\Gamma(2h)\sin{\pi h}.
		\end{equation}
		In particular, $\xi_{B}$ is hyperuniform if and only if $h<1/2$.
	\end{proposition}
	Recall that for the limit case $ h = 1/2$, the fBm is a classical Brownian motion, it has independent increments. On the other hand, for $ h<1/2$, the increments are negatively correlated and negative correlation is typical of hyperuniformity. The exponent $1-2h$ that appears in the first term of the asymptotic expansion of $s_{\xi_{B}}$ is known as the degree of hyperuniformity of the process. The latter controls the growth of the number variance and for the perturbed Palm lattice, one has
	\begin{equation*}
		\frac{\var{\xi_{B}(\mathsf{B}_{r})}}{\Vol{\mathsf{B}_{r}}} \underset{r\to+\infty}{\sim}r^{2h},
	\end{equation*}
	up to some constant factor of proportionality, see \cite[Proposition 2.2]{LachiezeRey2025a}. In particular, the higher $h$ is, the more disordered the system is, until it is no longer hyperuniform for $h\geq 1/2$. In the next paragraph, we give a simple heuristic to derive the variance exponent before giving a formal proof in Section \ref{sec:proof_hyperuniformity}.
	
	\subsection{Qualitative derivation of hyperuniformity}
	\label{sec:heuristics-intro}
	
	Imagine for a brief moment that the function $f(x)=x+B_{x}$ is increasing.
	In this case, the number of points in $\Z_{B}\cap\mathsf{B}_{r}$
	is given by $x_{+}-x_{-}$, where 
	\begin{eqnarray}
		x_{-} & = & \min\left\{ n:f(n)>-r\right\} ,\\
		x_{+} & = & \max\left\{ n:f(n)<r\right\} .
	\end{eqnarray}
	The variable $x_{+}$ has fluctuations of the order $r^{h}$: for
	$a>0$,
	\begin{eqnarray}
		\P\left[x_{+}\in[r-ar^{h},r+ar^{h}]\right] & = & \P\left[f(r-ar^{h})<r,f(r+ar^{h})>r\right]\\
		& = & \P\left[B_{r-ar^{h}}<ar^{h},B_{r+ar^{h}}>-ar^{h}\right].
	\end{eqnarray}
	This probability is indeed small when $a$ is small and large when
	$a$ is large, uniformly in $r$. Similarly $x_{-}$ also fluctuates as $r^{h}$, so $\left|\Z_{B}\cap\mathsf{B}_{r}\right|$
	has variance of order $r^{2h}$ as predicted in Proposition \ref{prop:degree_hyperuniformity_palm_perturbed_lattice}. Finally, we note that although $f$ is not monotone, the probability
	$\P\left(f(x+C)<f(x)\right)$ remains very small for finite $C$ (that does not diverge with $r$). Therefore, the error introduced by assuming monotonicity is of order 1 and can be neglected.
    
	\subsection{Numerical simulations}
	
	The heuristic of the preceding section is backed with some numerical simulations. In dimension $1$, one of the main benefits of the perturbed Palm lattice lies in its computational efficiency as sampling $n$ points of a Gaussian process with stationary increments can be accomplished in $n\log(n)$ complexity, see \cite{Dietrich1997,Diecker2004}. On a personal note, the simulation of the Palm process $\Z_{B}$ with $n=2^{20}$ points in Figure \ref{fig:simulation} takes approximately $200$ milliseconds on the author's personal computer.
	
	\begin{center}
		\begin{figure}[H]
			\includegraphics[scale=0.28]{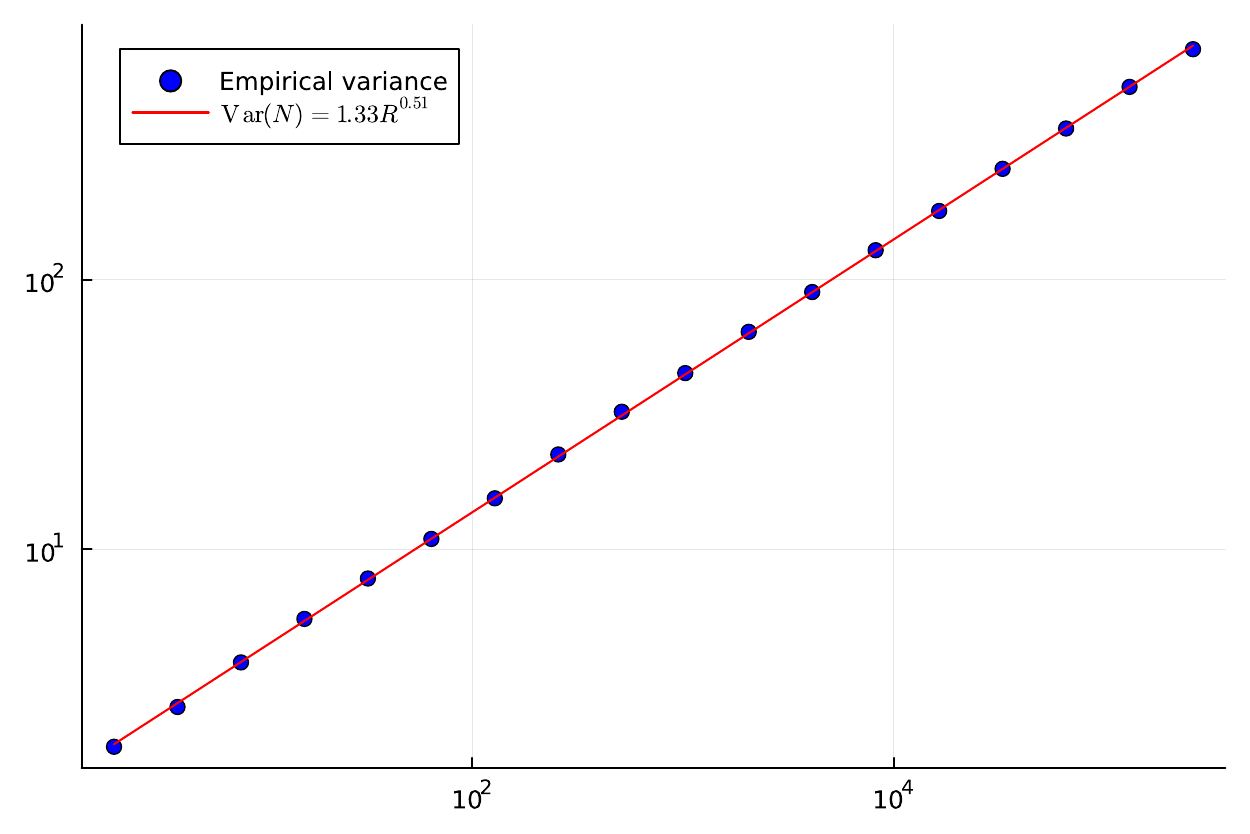}
			\includegraphics[scale=0.28]{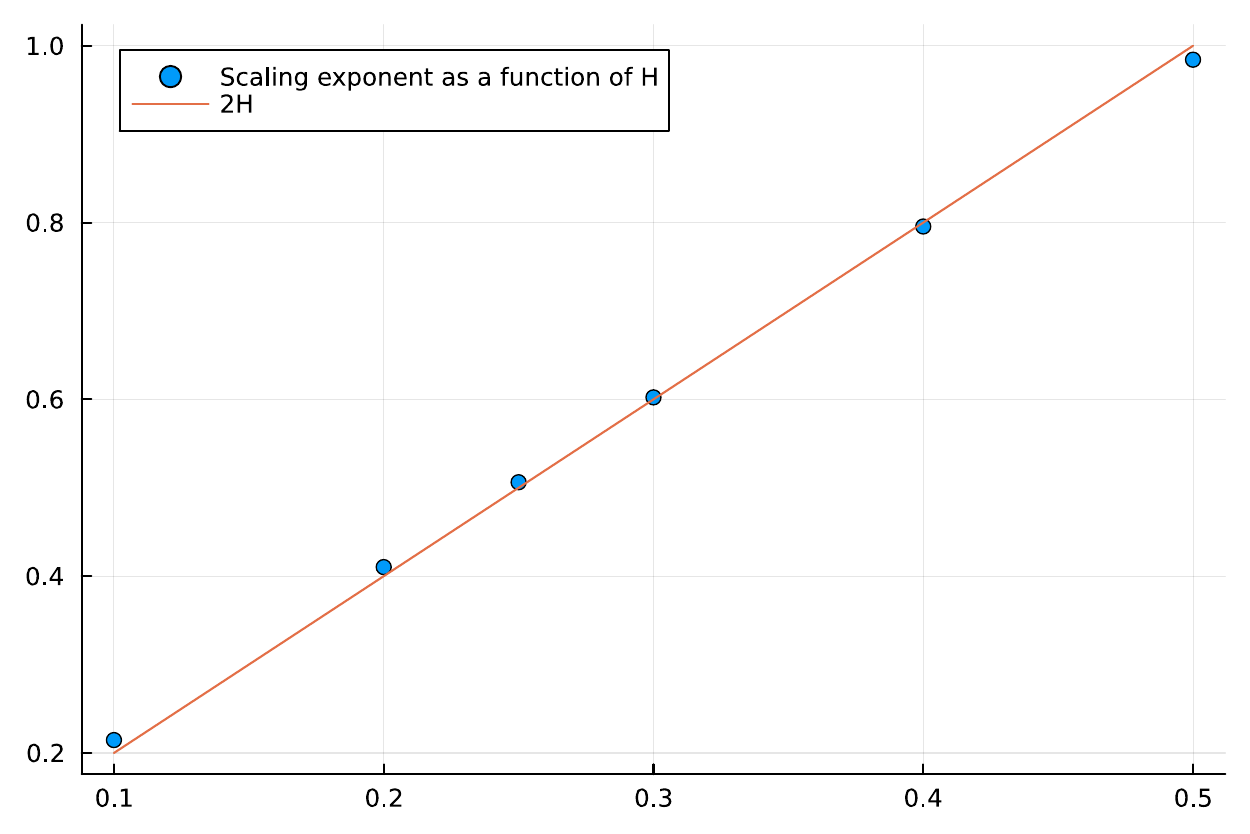}
			\caption{Log-regression of hyperuniformity exponent of $\xi_{B}$.}
		\end{figure}\label{fig:simulation}
		\tiny
		\justifying
		\textit{Description:} Hyperuniformity of the process $\xi_{B}$ where
		$B$ is a fractional Brownian motion. To the
		left we see in a log-log plot the variance of the number of particles
		in a ball of radius $r$ as a function of $r$, measured with $h=0.25$
		over $10000$ realizations of a system with $2^{20}$ points; linear
		regression yields the exponent $0.51$. Repeating this measurements
		for different values of $h$ provides the plot to the right; we compare
		the measured values to the theoretical exponent $1-2h$.
	\end{center}
	
	To simulate the underlying stationary point process rather than its Palm distribution, it is required to "dePalmize" $\hat{\xi}_{B}$. There are several algorithms which accomplish this task, and for theoretical purposes, we shall mention balanced allocations, see \cite[Chapters 10,11]{Last2017}. Yet, when it comes to numerical simulations, one should avoid relying on stable allocations as they are costly to compute, even though they allow for an exact simulation of $\xi_{B}$, . To circumvent this problem, we shall recall that $\xi_{B}$ is ergodic. In that case, it is well-known, see \cite[Proposition 12.5.VII]{Daley2008}, that 
	\begin{equation*}
		U_{x}+\hat{\xi}_{B} \underset{x\to +\infty}{\longrightarrow} \xi
	\end{equation*}
	in total variation, where $U_{x}$ is uniform on $[-x,x]$, independent of $\hat{\xi}_{B}$. Consequently, for $x$ \textit{large} enough, the law of $U_{x}+\hat{\xi}$ should provide a decent approximation of $\xi$. When $\xi_{B}$ is mixing  (see Section \ref{sec:mixing}), the preceding convergence can be strengthened to
	\begin{equation*}
		x+\hat{\xi}_{B} \underset{x\to +\infty}{\longrightarrow} \xi
	\end{equation*}
	the convergence holding in distribution; see \cite[Proposition 13.4.7]{Daley2008}.	

	\subsection{Open questions} 	
	In this section, we list some open question regarding perturbed Palm lattices.
	
	\subsubsection{Mixing properties}\label{sec:mixing}
	
	While it is known that $\xi_{B}$ is ergodic, little is known regarding its mixing properties. By mixing, we refer to a process which exhibits spatial asymptotic independence, in the sense that 
	\begin{equation*}
		\cov{f(\xi_{B})}{g(\tau_{t}\xi_{B})} \underset{\|t\|\to+\infty}{\longrightarrow} 0
	\end{equation*}
	for any measurable functions $f,g$. The absolutely continuous spectrum still gives some sort of asymptotic independence. When both $f(\xi_{B})=\xi_{B}(\varphi)$ and $g(\xi_{B})=\xi_{B}(\psi)$ are linear statistics,   equation \eqref{eq:phase_equation} yields
		\begin{equation*}
			\cov{\xi_{B}(\varphi)}{\tau_{t}\xi_{B}(\psi)} =\lambda_{\xi} \int_{\R^{d}}\overline{\hat{\varphi}(x)}{\hat{\psi}(x)} e^{i(t,x)}\, s_{\xi_B}(x)\, \mathrm{d}x
		\end{equation*}
		and the latter converges to $0$ as $\|t\|\to +\infty$ by the Riemann-Lebesgue lemma.   This is in contrast with lattices perturbed by  i.i.d variables (see introduction), for which even this weak mixing does not hold.
	\medbreak

	This question is not easily ascertainable for general functionals as, to the author's knowledge, there are no intrinsic criteria for mixing depending solely on the Palm distribution of the process. This question differs greatly from the question of ergodicity, since it can be proved that a stationary point process is ergodic if and only if its Palm distribution is, with respect to  some transformation known as point stationarity, see \cite[Chapter 13.4, Exercice 13.4.7]{Daley2008} and \cite{Thomassey2025}. Unfortunately, the equivalence does no longer holds when ``ergodic'' is replaced with ``mixing'' since the shifted lattice is not mixing, but its Palm distribution is. We still conjecture $ \xi _{ B}$ to be mixing since the underlying $d$-fBm is mixing (see Corollary \ref{cor:mixing_fBf}), hence exhibit long-range independence, and we expect this property to carry over.
	
	\subsubsection{Number-rigidity and deletion-tolerance} 
	
	A second open problem is related to the number-rigidity of $\xi_{B}$. Say that $\xi_{B}$ is number-rigid if the number of points of $\xi_{B}$ falling into a bounded non-empty open set $O\subset \R^{d}$ depends only on the outer configuration $\xi_{B}^{out} = \{x\in \xi_{B} : x\not\in O\}$. In dimension $1$, it is known that a large class of hyperuniform point processes exhibits some form of number-rigidity, known as linear rigidity, see \cite{Ghosh2016,Ghosh2017,Ghosh2021,LachiezeRey2025}. Since
		\begin{equation*}
			\int_{\mathsf{B}_{\varepsilon}(0)} \frac{1}{s_{\xi_{B}}(t)}\,\mathrm{d}t < +\infty,
		\end{equation*}
		it is known \cite[Formula (8)]{LachiezeRey2025} that  $\xi_{B}$ is not linearly number-rigid. Since linear-rigidity is the one observed in most cases, we expect $\xi_{B}$ to be non-rigid even in the hyperuniform case $h<1/2$. One way to disprove rigidity is to prove that $\xi_{B}$, or equivalently its Palm version $\hat{\xi}_{B}$, is deletion-tolerant (see \cite{Holroyd2013}), meaning that one cannot distinguish $\hat{\xi}_{B}$ from the process obtained by removing one of its atoms. We attempted this approach, but it was unsuccessful.
	
	\section{Generalities on point processes and fractional Brownian fields}\label{sec:stationarity_stochastc_increments}
	
	\subsection{Point processes}
	
	A random  point process $\xi=\{\xi_{n}: n\in\Z\}$ is a random variable in the space of locally finite random configurations of points in $\R^{d}$ equipped with the $\sigma$-algebra generated by the sets $\A_{K}=\{\xi \subset\R^{d} : \card{\xi\cap K}<+\infty \}$ for $K$ compact. Alternatively, identifying $\xi_{n}$ with an atom of a discrete measure, $\xi$ can be interpreted as a locally finite random measure, legitimating the use of the notation $\xi(\varphi)$ denote the linear statistics, \textit{i.e.}
	\begin{equation*}
		\xi(\varphi) = \sum_{n \in \Z}\varphi(\xi_{n}) = \int_{\R^{d}}\varphi(t)\, \xi(\mathrm{d}t).
	\end{equation*}
	
	As highlighted in the introduction, we are interested in those point processes which are stationary, that is invariant under the group of translations $(\tau_{t})_{t\in \R^{d}}$. Since a stationary point process induces a measure-preserving system, it is possible to study the ergodic properties of $\xi$ and we will say that  $\xi$ is said ergodic if
	\begin{equation*}
		\P(\xi \in A) \in \{0,1\}
	\end{equation*}
	for any invariant event $A$, \textit{i.e.} $\tau_{t}A=A$ for all $t\in \R^{d}$. Following this definition, it is not hard to check that the shifted lattice $\xi =\Z^{d}+U$ is ergodic.\medbreak
	
	Provided the intensity $\lambda_{\xi}=\esp{\xi([0,1]^{d})}$ is finite and non-zero, the distribution of $\xi$ is identified by its Palm distribution $\hat{\xi}=\{\hat{\xi}_{n} : n \in \Z\}$. The latter is a point process containing a.s. $0$ as an atom and is characterized by Campbell's formula 
	\begin{equation}\label{eq:campbell formula general}
		\esp{\int_{\R^{d}}f(\tau_{t}\xi,t) \, \xi(\mathrm{d}t)} = \lambda_{\xi}\esp{\int_{\R^{d}}f(\hat{\xi},t)\, \mathrm{d}t}.
	\end{equation}
	The latter allows to express the distribution of $\hat{\xi}$ in terms of $\xi$, namely
	\begin{equation}\label{eq:campbell formula}
		\esp{f(\hat{\xi})} = \frac{1}{\lambda_{\xi}}\esp{\int_{[0,1]^{d}}f(\tau_{t}\xi)\, \xi(\mathrm{d}t)}
	\end{equation}
	and vice versa
	\begin{equation}\label{eq:inverse_campbell formula}
		\esp{f(\xi)} = \lambda_{\xi}\esp{\int_{V_{0}(\hat{\xi})}f(\tau_{t}\hat{\xi})\, \mathrm{d}t}.
	\end{equation}
	Here $V_{0}(\hat{\xi}) = \{x \in \R^{d} : \forall n \in \Z, \|x\|^{2}\leq \|x-\hat{\xi}_{n}\|^{2}\}$ is the $0$-Voronoi cell.\medbreak

	When $\xi$ has a finite second moment measure, \textit{i.e.}
	\begin{equation*}
		\esp{\xi(K)^{2}}<+\infty, \quad K \text{ compact},
	\end{equation*}
	Campbell's formula \eqref{eq:campbell formula} allows to express the variance of a linear statistics $\varphi \in \mathcal{S}(\R^{d})$ as
	\begin{equation}\label{eq:variance of linear statistics}
		\var{\xi(\varphi)} = \lambda_{\xi}\esp{\int_{\R^{d}}\varphi \star \varphi(t)\, \mathcal{C}_{\xi}(\mathrm{d}t)}
	\end{equation}
	where 
	\begin{equation*}
		\varphi \star \varphi(t) = \int_{\R^{d}}\varphi(s)\varphi(s+t)\, \mathrm{d}s,
	\end{equation*}
	and
	\begin{equation}\label{eq:correlation measure}
		\mathcal{C}_{\xi}(\varphi) = \esp{\int_{\R^{d}}\varphi(t)\, \hat{\xi}(\mathrm{d}t)}-\lambda_{\xi}\int_{\R^{d}}\varphi(t)\, \mathrm{d}t
	\end{equation}
	is known as the correlation measure of the process. Since $\mathcal{C}_\xi$ is the difference of two possibly infinite measures, it is not necessarily a signed measure. Yet, it induces a tempered distribution whose Fourier transform is exactly the Bartlett's spectrum $\mathcal{S}_{\xi}$, see  \cite[Chapter 8]{Daley2008}. $\mathcal{S}_{\xi}$ is then necessarily a $\sigma$-finite positive measure, and when it has a density $s_{\xi}$ with respect to Lebesgue measure, the latter is called the  {\it structure factor} of the process. In particular, we recover
	\begin{equation}\label{eq:variance of linear statistics_furier}
		\var{\xi(\varphi)} = \lambda_{\xi}\esp{\int_{\R^{d}}|\hat{\varphi}(t)|^{2} \, \mathcal{S}_{\xi}(\mathrm{d}t)}
	\end{equation}
	where 
	\begin{equation*}
		\hat{\varphi}(t)=\int_{\R^{d}}\varphi(t)e^{-i( t,x)}\, \mathrm{d}x
	\end{equation*}
	
	\subsection{Fractional Brownian fields}\label{sec:fbm}
	
	In this section, we briefly recall some properties of Gaussian processes with stationary increments. In what follows, we always assume that $B$ is a centered $\R^{d}$-valued Gaussian process $B=(B^{(1)}, \dots, B^{(d)})$ with independent coordinates and stationary increments. Under these assumptions, the distribution of $B$ is determined by its variogram, which we shall always assume to be continuous,
	\begin{equation*}
		\Sigma_{t}=\var{B_{t}}=\operatorname{Diag}(v_{1}(t), \dots, v_{d}(t)), \quad t\in \R^{d},
	\end{equation*}
	where $v_{i}$ is the individual variogram of $B^{(i)}$. Similarly to Bochner's theorem which characterizes the continuous covariance function of stationary Gaussian process, there is a spectral characterization of continuous variograms, known as the Levy-Khintchine theorem, see \cite[Theorem 4.12]{Schilling2009}. The latter allows to rewrite	each individual variogram $v_{i}$ as 
	\begin{equation}\label{eq:levy_khintchine}
		v_{i}(t)=a_{i}t^{2}+\int_{\R^{d}}|1-e^{-i( t,x)}|^{2} \, \mu_{i}(\mathrm{d}x), \quad t\in \R^{d},
	\end{equation}
	for some drift parameter $a_{i}\geq 0$ and $\mu_{i}$ a positive symmetric Borel measure  satisfying
	\begin{equation*}
		\int_{\R^{d}} (1\wedge  \|x\|^{2})\, \mu_{i}(\mathrm{d}x) <+\infty.
	\end{equation*}
	$\mu_{i}$ is known as the Levy measure of the process.\medbreak

	In this article, we will only consider the $d$-fBf defined in the introduction, see Section \ref{sec:introduction}. In that case, $a_i=0$ and the Levy measures $\mu_{i}$ are absolutely continuous with respect to Lebesgue measure, with 
	\begin{equation*}
		\mu_{i}(\mathrm{d}t)=\frac{C_{i}\mathrm{d}t}{\|t\|^{d+2h}}.		
	\end{equation*} 
	for some constants of proportionality $C>0$.\medbreak
	
	This observation is important since the family of Levy measure $\mu=(\mu_{1}, \dots, \mu_{d})$ completely characterizes the ergodic properties of $B$. Before stating this result, let us recall some terminology. Say that $B$ is 
	\begin{enumerate}
		\item \textbf{ergodic} if
		\begin{equation*}
			\P(B\in A)\in\{0,1\}
		\end{equation*}
		for any $\theta_{t}$ invariant event $A$,  
		\item \textbf{weakly-mixing} if
		\begin{equation*}
			\frac{1}{\Vol{\mathsf{B}_{r}}}\int_{\mathsf{B}_{r}}\left|\P(B\in \theta_{t}A, B \in A')-\P(B\in A)\P(B\in A')\right|\, \mathrm{d}t \underset{r\to +\infty}{\longrightarrow} 0
		\end{equation*}
		for any pair of events $A,A'$.
		\item \textbf{(strongly)-mixing} if
		\begin{equation*}
			\P(B\in \theta_{t}A, B \in A') \underset{\|t\|\to+\infty}{\longrightarrow}\P(B\in A)\P(B\in A')
		\end{equation*}
		for any pair of events $A,A'$.
	\end{enumerate}  
	
	\begin{theorem}(Maruyama for Gaussian processes with stationary increments)\label{thm:Maruyama_stationary_increments}
		Let $B$ be a $\R^{d}$-valued Gaussian process with stationary increments and independent coordinates, with $a_{i}=0$ for $1\leq i \leq d$.
		\begin{itemize}
			\item $B$ is ergodic and weakly-mixing if and only if $\mu_{i}$ is non-atomic for any $1\leq i \leq d$,
			\item $B$ is mixing if and only if
			\begin{equation*}
				V_{a,b}(t) \underset{\|t\|\to +\infty}{\longrightarrow}  0, \quad a,b \in \R^{d},
			\end{equation*}
			where 
			\begin{align*}
				V_{a,b}(t)&=\cov{B_{-a}}{ B_{t+b}-B_{t}}\\
				&=\Sigma_{t+a}+\Sigma_{t+b}-\Sigma_{t+a+b}-\Sigma_{t}.
			\end{align*}
		\end{itemize}
		In particular, if $\mu_{i}$ is absolutely continuous for $1\leq i \leq d$, then $B$ is mixing.
	\end{theorem}
	
	The standard version of Maruyama's theorem is usually stated for stationary Gaussian processes, see \cite{Maruyama1949}. Nonetheless, its proof can be easily adapted to Gaussian processes with stationary increments and we refer to Appendix \ref{ap:Gaussian_proces_stationary_increments} for a proof of this fact.
	
	\begin{corollary}\label{cor:mixing_fBf}
		The $d$-fBm is a mixing Gaussian process with stationary increments. 
	\end{corollary}

	\section{Proof of Theorem \ref{thm:fBm peturbed lattice}}
	
	In this section, we consider $\xi$ a stationary and ergodic point process with finite second moment measure and intensity $\lambda_{\xi}$.  Recall that
	\begin{equation*}
		\hat{\xi}_{B}=\{x+B_{x} :  x \in \hat{\xi} \}.
	\end{equation*} 
	The goal of what follows is to prove Theorem \ref{thm:fBm peturbed lattice}, which asserts that 
	
	\begin{enumerate}[(1)]
		\item \label{item:thm 1} $  \hat \xi _{ B}$ is the Palm distribution of a stationary  simple point process $\xi_{B}$,
		\item \label{item:thm 2}$\xi_{B}$ is ergodic,
		\item \label{item:thm 3} $\xi_{B}$ has intensity $\lambda_{\xi_{B}}= \lambda_{\xi}$,
		\item \label{item:thm 4} the structure factor of $\xi_{B}$ is 
		\begin{equation*}
			s_{\xi_{B}}(t) = \esp{\int_{\R^{d}}e^{-\frac{1}{2}(\Sigma_{x}t,t)}e^{-i( t,x)} \, \hat{\xi}(\mathrm{d}x)}, \quad t\in \R^{d}.
		\end{equation*}
	\end{enumerate}
	
	We will prove each of these points successively.
	
	\subsection{Proof of (\ref{item:thm 1})}
	
	The stationarity of $\xi_B$ follows from the main result of \cite{Thomassey2025}, but for completeness we shall briefly reprove it hereafter. Recall that $\hat{\xi}_B$ is a Palm measure if and only if $0 \in \xi_B$ and Mecke's invariance principle for Palm measure \cite[Theorem 13.2.VIII]{Daley2008},
	\begin{equation}
		\esp{\int_{\R^{d}}\varphi(\tau_{t}\hat{\xi}_B, -t) \, \hat{\xi}_B(\mathrm{d}t)} = \esp{\int_{\R^{d}} \varphi(\hat{\xi}_B, t) \, \hat{\xi}_B(\mathrm{d}t)}.
		\label{eq:Mecke}
	\end{equation}
	is satisfied for any measurable positive function $\varphi$. Since $0\in \hat{\xi}$ and $B_{0}=0$, $0\in \hat{\xi}_{B}$. For the proof of \eqref{eq:Mecke}, note that
	\begin{align*}
		\esp{\int_{\R^{d}}\varphi(\tau_{t}\hat{\xi}_B, -t) \, \hat{\xi}_B(\mathrm{d}t)} = \esp{\int_{\R^{d}} \varphi(\tau_{t+B_t}\hat{\xi}_{B},-t-B_t) \, \hat{\xi}(\mathrm{d}t)}
	\end{align*}
	and 
	\begin{align*}
		\tau_{t+B_t}\hat{\xi}_B &= \{x+B_x-t-B_t : x \in \hat{\xi}\}\\
		&=\{x-t+B_{x+t-t}-B_t : x\in \hat{\xi}\}\\
		&=\{x+B_{x+t}-B_{t} : x\in \tau_{t}\hat{\xi}\}\\
		&=\{x+\theta_{t}B_{x} : x\in \tau_{t}\hat{\xi}\}\\
		&=(\tau_{t}\hat{\xi})\star (\theta_{t}B)
	\end{align*}
	where $\xi\star B = \{x+B_x : x\in \xi\}$. As a consequence, 
	\begin{align*}
		\esp{\int_{\R^{d}}\varphi(\tau_{t}\hat{\xi}_B, -t) \, \hat{\xi}_B(\mathrm{d}t)} = \esp{\int_{\R^{d}}\varphi((\tau_{t}\hat{\xi})\star(\theta_{t}B), -t-B_{t}) \, \hat{\xi}(\mathrm{d}t)}.
	\end{align*}
	For a fixed $t$, $\theta_{t}B$ and $B$ have same distribution as $B$ has stationary increments, henceforth the mutual independence of $B$ and $\hat{\xi}$ yields
	\begin{align*}
		\esp{\int_{\R^{d}}\varphi(\tau_{t}\hat{\xi}_B, -t) \, \hat{\xi}_B(\mathrm{d}t)} &= \esp{\int_{\R^{d}}\varphi((\tau_{t}\hat{\xi})\star B, -t-\theta_{-t}B_{t}) \, \hat{\xi}(\mathrm{d}t)}\\
		&=\esp{\int_{\R^{d}}\varphi((\tau_{t}\hat{\xi})\star B, -t+B_{-t}) \, \hat{\xi}(\mathrm{d}t)}.
	\end{align*}
	Finally, since $\hat{\xi}$ is a Palm measure, it satisfies Mecke's invariance principle \cite[Theorem 13.2.VIII]{Daley2008},
	\begin{equation*}
		\esp{\int_{\R^{d}}\psi(\tau_{t}\hat{\xi}, -t) \, \hat{\xi}(\mathrm{d}t)} =  \esp{\int_{\R^{d}}\psi(\hat{\xi}, t) \, \hat{\xi}(\mathrm{d}t)}.
	\end{equation*} 
	Conditioning on $B$ and applying this result to $\psi(\hat{\xi}, t) = \varphi(\hat{\xi}\star B, t+B_{t}) $ yields ultimately 
	\begin{align*}
		\esp{\int_{\R^{d}}\varphi(\tau_{t}\hat{\xi}_B, -t) \, \hat{\xi}_B(\mathrm{d}t)} &= \esp{\int_{\R^{d}}\varphi(\hat{\xi}\star B, t+B_{t}) \, \hat{\xi}(\mathrm{d}t)}\\
		&=\esp{\int_{\R^{d}}\varphi(\hat{\xi}_{B}, t) \, \hat{\xi}_{B}(\mathrm{d}t)}.
	\end{align*}
	Hence \eqref{eq:Mecke} is satisfied and it follows that $\hat{\xi}_{B}$ is the Palm measure of a stationary point process. The latter can be renormalised into a Palm distribution, \cite[Proposition 3.10]{Thomassey2025} if and only if
	 
\begin{align}\label{eq:voronoi_cell_bound}
		\esp{\Vol{V_{0}(\hat{\xi}_{B})}}<+\infty,
\end{align}
	where $V_{0}(\hat{\xi}_{B}) = \{t \in \R^{d} : \forall x \in \hat{\xi}, \|t-x-B_x\| \leq \|x\|\}$ denotes the $0$-Voronoi cell of $\hat{\xi}_{B}$. Condition \eqref{eq:voronoi_cell_bound} is mostly a technical assumption related to Palm inversion formula \eqref{eq:inverse_campbell formula}. If it is not satisfied, $\hat{\xi}_{B}$ is still "in a certain sense" a Palm distribution, but this Palm distribution is no longer associated with a stationary point process defined on a probability space, but rather with a stationary point process defined on an infinite measured space, see \cite{Daley2008}. In that case, we speak of a Palm measure.\medbreak
	
	We prove  \eqref{eq:voronoi_cell_bound} for $d=1$, and then briefly sketch how to adapt the argument in higher dimension. Fix $T>0$ large enough so that both $\xi \cap (0,T]$ and $\xi \cap [-T,0)$ are non-empty with high probability. If $|B_T| \in (0,T)$ and $|B_{-T}|\in (0,T)$, then both intervals $(-2T,0)$ and $(0,2T)$ contain at least one point of $\hat \xi_B$. In particular, 
	\begin{equation*}
		\Vol{V_0(\hat{\xi}_B)}<4T.
	\end{equation*}
	Since the probabilities $|B_{\pm T}| \in (0,T)$ decay as stretched exponential, $\Vol{V_0(\hat \xi_B)}$ must have finite expectation.\medbreak
	
	In higher dimension, observe that if for all $x\in \{-T,T\}^d$ we have $|B_x|<T$, then each quadrant of $\R^d$ contains a point of $\hat \xi_B$ at distance at most $2T$ from the origin. In this case, the volume of $V_0(\hat \xi_B)$  is bounded by $T^d$ which as before implies that is has finite expected value. This proves \eqref{eq:voronoi_cell_bound} and a fortiori \eqref{item:thm 1}.
	
	\subsection{Proof of (\ref{item:thm 2})}
	 	Recall that, as established in \cite[Theorem 3.2]{Thomassey2025}, $\xi_{B}$ is ergodic provided that $\xi$ and $B$ are jointly ergodic. Since $\xi$ is ergodic and $B$ has a non-atomic Levy measure, implying that $B$ is weakly mixing (Corollary \ref{cor:mixing_fBf}), the result follows directly from \cite[Theorem 2.36]{Einsiedler2011}.

	\subsection{Proof of (\ref{item:thm 3})}
	
		The goal of this section is to evaluate the intensity of $\xi_{B}$ and prove that $\lambda_{\xi_{B}} = \lambda_{\xi}$. There are multiple ways to derive this result, but here we use ergodicity and a shift-coupling argument due to Thorisson.

	\begin{lemma}
		Let $\xi$ be an ergodic point process with intensity $\lambda_{\xi}$, then 
		\begin{equation*}
			\lim_{r\to +\infty}\frac{\xi(\mathsf{B}_{r})}{\Vol{\mathsf{B}_{r}}}  = \lim_{r\to +\infty}\frac{\hat{\xi}(\mathsf{B}_{r})}{\Vol{\mathsf{B}_{r}}} = \lambda_{\xi} \quad a.s.
		\end{equation*}
	\end{lemma}

	\begin{proof}
		Let $\mathrm{N}(\R^{d})$ denote the set of locally finite measures on $\R^{d}$ and consider the event 
		\[	
			A = \left\{
				\psi \in \mathrm{N}(\R^{d}) : \frac{\psi(\mathsf{B}_{r})}{\Vol{\mathsf{B}_{r}}} \underset{r\to+\infty}{\longrightarrow} \lambda_{\xi}
		\right\}.
		\] 
		Since $A$ is invariant under translation and $\xi$ is ergodic, the event $\xi \in A$ has probability $1$. By Campbell's  formula \eqref{eq:campbell formula}, it follows that
		\begin{equation*}
			\P(\hat{\xi} \in A) = \frac{1}{\lambda_{\xi}}\esp{\int_{[0,1]^{d}}\underbrace{\mathds{1}(\theta_{t}\xi \in A)}_{=1} \, \xi(\mathrm{d}t)}
			=1.
		\end{equation*} 
		This finshes the proof.
	\end{proof}

	Returning to the proof of $(3)$, the preceding lemma implies that
	\begin{equation*}
		\lambda_{\xi_B}=\lim_{r\to +\infty}\frac{\xi_{B}(\mathsf{B}_{r})}{\Vol{\mathsf{B}_{r}}} \quad a.s.
	\end{equation*}
	since $\xi_B$ is ergodic. Moreover, the ergodicity of $\xi_{B}$ ensures that $\xi_{B}$ and $\hat{\xi}_{B}$ can be shift-coupled (see \cite[Theorem 5.32]{Kallenberg2017}), that is there exists a random variable $T$ in $\R^{d}$ such that
	\begin{equation*}
		\xi_{B} \stackrel{(d)}{=}\tau_{T}\hat{\xi}_{B}.
	\end{equation*}
	As a consequence,
	\begin{align*}
		\lambda_{\xi_B} &= \frac{\hat{\xi}_{B}(T+\mathsf{B}_{r})}{\Vol{\mathsf{B}_{r}}}\\
		&=\frac{\card{x \in \hat{\xi} : x+B_{x}-T \in \mathsf{B}_{r}}}{\Vol{\mathsf{B}_{r}}} \quad a.s.
	\end{align*}
	
	Fix $\varepsilon>0$. Since $\|B_{x}\|=o(\|x\|)$ \textit{a.s.}, there exists a random radius $r_{0}(\varepsilon)>0$ such that 
	\begin{equation*}
		\frac{\hat{\xi}(\mathsf{B}_{(1-\varepsilon)r})}{\Vol{\mathsf{B}_{r}}}\leq \frac{\xi_{B}(\mathsf{B}_{r})}{\Vol{\mathsf{B}_{r}}} \leq  \frac{\hat{\xi}(\mathsf{B}_{(1+\varepsilon)r)}}{\Vol{\mathsf{B}_{r}}}.
	\end{equation*}
	for any $r>r_{0}(\varepsilon)$. Taking the limits on both sides yields
	\begin{equation*}
		(1-\varepsilon)^{d} \leq \lambda_{\xi_{B}}\leq (1+\varepsilon)^{d}
	\end{equation*}
	and we let $\varepsilon\to 0$ to finish the proof. 
	
	\subsection{Proof of (\ref{item:thm 4})}
	
	In this section, we compute the Bartlett's spectrum of $\xi_{B}$. Recall that the latter is well-defined for point processes whose second moment measures are finite. Hence, we need to check first that 
	\begin{equation}\label{eq:second_momemnt_perturbed_palm_lattice}
		\esp{\xi_{B}(K)^{2}}<+\infty, \quad \text{ $K$ compact}.
	\end{equation}

	The proof of \eqref{eq:second_momemnt_perturbed_palm_lattice} relies on Campbell's formula \eqref{eq:campbell formula general} since
	\begin{align*}
		\esp{\xi_{B}(K)^{2}}&=\esp{\int_{\R^{d}}\mathds{1}_{K}\star\mathds{1}_{K}(t) \,\hat{\xi}_{B}(\mathrm{d}t)}.
	\end{align*}
	
	The convolution $\mathds{1}_{K}\star\mathds{1}_{K}$ is a continuous function supported on $L=K-K$ compact. Hence, it is bounded by some constant $C>0$ and it follows from the independence between $\hat{\xi}$ and $B$ that 
	\begin{equation}
		\esp{\xi_{B}(K)^{2}}\leq C \esp{\int_{\R^{d}}\P(x+B_{x} \in L)\, \hat{\xi}(\mathrm{d}x)}.
	\end{equation}
	Since $B$ is a $d$-fBm whose variance is given by
	\begin{equation*}
		\Sigma_{x}=\begin{pmatrix}
				\|x\|^{2h_{1}}&&(0)\\
				&\ddots&\\
				(0)&&\|x\|^{2h_{n}}
			\end{pmatrix}
	\end{equation*}
	with $0<h_{i}<1$, it is clear that $x \longmapsto\P(x+B_{x} \in L)$ is exponentially decreasing. Hence, according to the following lemma, 
	\begin{equation*}
		\esp{\xi_{B}(K)^{2}}<+\infty.
	\end{equation*}

	\begin{lemma}
		Let $\xi$ be a stationary point process with finite second moment and $f : \R^{d} \longrightarrow \R_{+}$ be a measurable mapping satisfying 
		\begin{equation*}
			f(t) \underset{\|t\|\to +\infty}{=}O\left(\|t\|^{-(d+1)}\right)
		\end{equation*}
		Then, 
		\begin{equation*}
			\esp{\int_{\R^{d}}f(t) \, \hat{\xi}(\mathrm{d}t)} < +\infty. 
		\end{equation*}
	\end{lemma}

	\begin{proof}
		For $n\in \Z^{d}$, let $[n,n+1] = \prod_{i=1}^{d}[n_{i},n_{i}+1]$.\medbreak

		By Campbell's formula \eqref{eq:campbell formula},
		\begin{align*}
			\esp{\int_{\R^{d}}f(t) \, \hat{\xi}(\mathrm{d}t)} &= \frac{1}{\lambda_{\xi}}\esp{\int_{[0,1]^{d}}\int_{\R^{d}} f(t-x) \, \xi(\mathrm{d}t)\, \xi(\mathrm{d}x)}\\
			&\leq C\esp{\xi([0,1]^{d})\int_{\R^{d}}\frac{1}{1+\|t\|^{d+1}}\, \xi(\mathrm{d}t)}\\
			&=C\sum_{n \in \Z^{d}}\frac{1}{1+\|n\|^{d+1}}\esp{\xi([0,1]^{d})\xi([n,n+1]^{d})}\\
			\text{(Cauchy-Schwarz)}&\leq C\sum_{n \in \Z^{d}}\frac{1}{1+\|n\|^{d+1}}\esp{\xi([0,1]^{d})^{2}}\\
			&<+\infty,
		\end{align*}
		where $C$ is a positive constant allowed to change from line to line.
	\end{proof}

	Following the finiteness of the second moment measure, the correlation measure $\mathcal{C}_{\xi_{B}}$ of $\xi_{B}$ is well-defined, and following equation \eqref{eq:correlation measure}, its explicit formula is given by the tempered distribution 
	\begin{equation}\label{eq:palm_lattice_correlation}
		( \mathcal{C}_{\xi_{B}}, \varphi) = \esp{\int_{\R^{d}}\varphi(t)\, \hat{\xi}(\mathrm{d}t)} - \lambda_{\xi}\int_{\R^{d}}\varphi(t)\, \mathrm{d}t.
	\end{equation}
	
	Recalling that the Bartlett's spectrum is the Fourier transform of the covariance function, one has 
	
	\begin{equation}\label{eq:structure_factor_perturbed_palm_lattice}
		\begin{aligned}
			\mathcal{S}_{\xi_B}(\varphi) &= \esp{\int_{\R^{d}}\int_{\R^{d}}\varphi(t)e^{-i(x, t)-i(B_{x},t)} \, \mathrm{d}t \, \hat{\xi}(\mathrm{d}x)}-\lambda_{\xi}\int_{\R^{d}}\hat{\varphi}(t)\, \mathrm{d}t\\
			&=\esp{\int_{\R^{d}}\int_{\R^{d}}\varphi(t)e^{-i( x, t)-\frac{1}{2}( \Sigma_{x}t,t) }\, \mathrm{d}t\, \hat{\xi}(\mathrm{d}x)}-\lambda_{\xi}(2\pi)^{d}\varphi(0).
		\end{aligned}
	\end{equation}

	The preceding formula does not allow to see that $\mathcal{S}_{\xi_B}$ is absolutely continuous with respect to Lebesgue measure. In order to do so, we decompose $\mathcal{S}_{\xi_{B}}$ into an atomic and non-atomic part, namely
	\begin{equation*}
		\mu = \mathcal{S}_{\xi_{B}}(\{0\})
	\end{equation*}
	and 
	\begin{equation*}
		\mathcal{T}_{\xi_{B}} = \mathcal{S}_{\xi_{B}} - \mu \delta_{0}.
	\end{equation*}
		
	The decomposition $\mathcal{S}_{\xi_B}=	\mathcal{T}_{\xi_{B}}+\mu \delta_{0}$ is simply the Lebesgue decomposition of $\mathcal{S}_{\xi}$ with respect to to $\delta_{0}$ and we are left proving that
	\begin{enumerate}[(a)]
		\item\label{item:absolutely_continuous_part} $\mathcal{T}_{\xi_{B}}$ is absolutely continuous with respect to Lebesgue measure,
		\item \label{item:atom}$\mu=0$.
	\end{enumerate}
	
	First, we study the contribution of $\mathcal{T}_{\xi_{B}}$. Let $\varphi \in \mathcal{S}(\R^{d})$ be a Schwartz function supported on the complement of $\mathsf{B}_{\varepsilon}$. Equation \eqref{eq:structure_factor_perturbed_palm_lattice} yields
	
	\begin{align}\label{eq:absolutely_continuous_part}
		\mathcal{S}_{\xi_{B}}(\varphi) &=\mathcal{T}_{\xi_{B}}(\varphi)= \int_{\R^{d}}\esp{e^{-\frac{1}{2}( \Sigma_{x}t,t)-i ( t,x)} \, \hat{\xi}(\mathrm{d}x)} \varphi(t)\,\mathrm{d}t,
	\end{align}
	
	the inversion sum-integral being justified with Fubini since
		\begin{align*}
			\left|e^{-\frac{1}{2}( \Sigma_{n}t,t)-i( t,n)} \varphi(t)\right| &=e^{-\frac{1}{2}\sum_{i=1}^{d}h_{i}t_{i}^{2}}|\varphi(t)|\\
			& \leq e^{-\frac{1}{2}h_{\star}^{2}\varepsilon^{2}}|\varphi(t)|
		\end{align*}

		which is integrable. Here, $h_{\star}=\min\{h_{i}: 1 \leq i \leq d\}$.\medbreak

		As a consequence, $\mathcal{T}_{\xi_{B}}$ is absolutely continuous with respect to the Lebesgue measure and its density is given by 
		\begin{equation*}
			s_{\xi}(t)=\esp{e^{-\frac{1}{2}( \Sigma_{x}t,t)-i ( t,x)} \, \hat{\xi}(\mathrm{d}x)}.
		\end{equation*}
	
	For the contribution of the atomic part at $0$, consider the Gaussian density
	\begin{equation*}
		\varphi_{\sigma}(t) = e^{-\frac{1}{2}\sigma^{2}\|t\|^{2}}.
	\end{equation*}
	Since $\varphi_{\sigma}$ converges pointwise to $\mathds{1}_{\{0\}}$ as $\sigma \to +\infty$, the dominated convergence theorem yields
	\begin{equation}\label{eq:strcuture_factor_atomic_part}
		\mathcal{S}_{\xi_{B}}(\varphi_{\sigma}) \underset{\sigma \to +\infty}{\longrightarrow} \mu.
	\end{equation}
	
	On the other hand, equation \eqref{eq:structure_factor_perturbed_palm_lattice} yields
	\begin{equation*}
		\mathcal{S}_{\xi_{B}}(\varphi_{\sigma})
		=\esp{\int_{\mathbb{R}^{d}}\int_{\R^{d}} e^{-\frac{1}{2}(\sigma^{2}t+\Sigma_{x}t,t)-i( t, x)} \, \hat{\xi}(\mathrm{d}x)\, \mathrm{d}t} - \lambda_{\xi}(2\pi)^{d}.
	\end{equation*}
	The inner integral is, up to a proportionality constant, the characteristic function of a Gaussian random variable with variance $(\sigma^{2} \Id_{d}+\Sigma_{n})^{-1}$. Write 
	\begin{equation*}
		(\sigma^{2} \Id_{d}+\Sigma_{n})^{-1} = \frac{1}{\sigma^{2}}(\Id_{d}+\varepsilon_{\sigma}^{n}),
	\end{equation*} 
	where $\varepsilon_{\sigma}^{x}$ is a symmetric matrix satisfying $\displaystyle{\lim_{\sigma \to +\infty}\varepsilon_{\sigma}^{n}=0}$. Then,
	\begin{equation*}
		\mathcal{S}_{\xi_{B}}(\varphi_{\sigma})=\frac{(2\pi)^{d/2}}{\sigma^{d}}\esp{\int_{\R^{d}}\frac{e^{-\frac{1}{2\sigma^{2}}\left(\|x\|^{2}+(\varepsilon_{\sigma}^{x}x, x)\right)}}{\sqrt{\det(\Id_{d}+\varepsilon_{\sigma}^{x})}} \, \hat{\xi}(\mathrm{d}x)}- \lambda_{\xi}(2\pi)^{d}.
	\end{equation*}
	
	Letting $\sigma\to+\infty$ yields
	\begin{align*}
		\frac{(2\pi)^{d/2}}{\sigma^{d}}\esp{\int_{\R^{d}}\frac{e^{-\frac{1}{2\sigma^{2}}\left(\|x\|^{2}+(\varepsilon_{\sigma}^{x}x, x) \right)}}{\sqrt{\det(\Id_{d}+\varepsilon_{\sigma}^{x})}} \, \hat{\xi}(\mathrm{d}x)}&\underset{\sigma \to +\infty}{\sim} \frac{(2\pi)^{d/2}}{\sigma^{d}}\esp{\int_{\R^{d}} e^{-\frac{\|x\|^{2}}{2\sigma^{2}}}\, \hat{\xi}(\mathrm{d}x)}\\
		\text{(Lemma \ref{lem:technical_lemma_integrability})}&\underset{\sigma \to +\infty}{\sim} \lambda_{\xi}\int_{\R^{d}} e^{-\frac{\|x\|^{2}}{2}} \, \mathrm{d}x\\
		&=\lambda_{\xi}(2\pi)^{d},
	\end{align*}
	so that
	\begin{equation}\label{eq:strcuture_factor_atomic_part_1}
		\mathcal{S}_{\xi_{B}}(\varphi_{\sigma}) \underset{\sigma\to +\infty}{\longrightarrow} 0.
	\end{equation}
	It follows from \eqref{eq:strcuture_factor_atomic_part} and \eqref{eq:strcuture_factor_atomic_part_1} that $\mu=0$.\medbreak	
	
	To finish the proof, it remains to state and prove Lemma \ref{lem:technical_lemma_integrability}, the main ingredient of the preceding derivation.

	\begin{lemma}\label{lem:technical_lemma_integrability}
		Let $\xi$ be a stationary and ergodic point process with finite intensity $\lambda_{\xi}$. Then, 
		\begin{equation*}
			\lim_{\sigma \to +\infty}\frac{1}{\sigma^{d}}\esp{\int_{\R^{d}}f\left(\frac{x}{\sigma}\right) \, \hat{\xi}(\mathrm{d}x)} = \lambda_{\xi}\int_{\R^{d}} f(x)\, \mathrm{d}x.
		\end{equation*}
		for any continuous and integrable $f : \R^{d}\longrightarrow \R$.
	\end{lemma}

	\begin{proof}
		Let $f(x)=\mathds{1}(x \in [a,b])$ with $[a,b]=\prod_{i=1}^{d}[a_i,b_i]$. Then
		\begin{equation*}
			\frac{1}{\sigma^{d}}\int_{\R^{d}} f(x) \, \hat{\xi}(\mathrm{d}x) = \frac{1}{\sigma^{d}}\hat{\xi}([\sigma a, \sigma b])
		\end{equation*}
        Since $\xi$ is ergodic, Thorisson's shift-coupling lemma \cite[Theorem 5.32]{Kallenberg2017} yields the existence of a random variable $T$ for which 
        \begin{equation*}
            \hat{\xi}\stackrel{(d)}{=} \xi-T.
        \end{equation*}
        Consequently,
        \begin{equation*}
			\frac{1}{\sigma^{d}}\int_{\R^{d}} f(x) \, \hat{\xi}(\mathrm{d}x) = \frac{1}{\sigma^{d}}\xi([T+\sigma a, T+\sigma b])
		\end{equation*}
        and the latter converges by the ergodic theorem (see \cite[Theorem 12.2.IV]{Daley2008}) a.s. and in $\mathbb{L}^{1}(\P)$ to $\Vol{[a,b]}$. The conclusion of Lemma \ref{lem:technical_lemma_integrability} follows by approximating a continuous and integrable $f$ with step functions.
        
	\end{proof}

	\section{Proof of Proposition \ref{prop:degree_hyperuniformity_palm_perturbed_lattice}}\label{sec:proof_hyperuniformity}
	
	We now concentrate on the dimension $d=1$ and study the first-term asymptotic at $0$ of the structure factor of the perturbed Palm lattice where $\hat{\xi}=\Z$, and $B$ is a fBm with Hurst index $h \in (0,1)$. According to Theorem \ref{thm:fBm peturbed lattice}, $\xi_{B}$ is a stationary ergodic point process with intensity $1$ and whose structure factor is given by
	\begin{equation*}
		s_{\xi_B}(t) = \sum_{n\in \Z}e^{-\frac{1}{2}|n|^{2h}t^{2}}e^{-int}.
	\end{equation*}
	
	To derive the required asymptotic, one may follow the intuitive idea that the structure factor $s_{\xi_B}(t)$ is well-approximated by the integral
	\begin{equation*}
		\bar{s}_{\xi_{B}}(t) = \int_{\R}e^{-\frac{1}{2}t^{2}|x|^{2h}}e^{-itx}\,\mathrm{d}x.
	\end{equation*}
	
	The latter is, up to a change of variables, the Fourier transform of the function $x \longmapsto e^{-\frac{1}{2}|x|^{2h}} $. Equivalently, up to the multiplicative factor $1/\pi$, it coincides with the density of an $\alpha$-stable distribution. The asymptotic behavior of such densities is well understood, see for instance, 
	\cite[Theorem 2.5.1]{Zolotarev2001}. In particular, for $h \leq 1$ and as $t \to 0$,
	\begin{equation*}
		\bar{s}_{\xi_{B}}(t) \sim \alpha_h |t|^{1-2h},
	\end{equation*}
	where the constant $\alpha_h$ is given in \eqref{eq:asymptotic_structure_factor_constant} . For $h > 1$, the same asymptotic relation remains valid and can be established via the mean of Mellin transform, see \cite[Chapter 6]{Bleistein1986}.\medbreak

	Hence, it remains to prove that $\bar{s}_{\xi_{B}}$ is a good approximation of $s_{\xi_B}$. We distinguish three cases depending on whether $h=1/2$, $h<1/2$ or $h>1/2$. 
	
	\subsection{Brownian motion $h=1/2$}
	
	For $h=1/2$, $s_{\xi_B}$ is explicit and its asymptotic can be computed explicitly.
	
	\begin{align*}
		s_{\xi_B}(t) &= \sum_{n \in \mathbb{Z}} e^{int} e^{-\frac{1}{2}t^{2}|n|}\\
		&=1+2\Re\left(\sum_{n \geq 1}e^{itn}e^{-\frac{1}{2}t^{2}n }\right)\\
		&=1+2\Re\left(\frac{e^{it-\frac{1}{2}t^{2}}}{1-e^{it-\frac{1}{2}t^{2}}}\right)\\
		&=1+o(1).
	\end{align*}
	
	\subsection{Hyperuniform case $h<1/2$}
	
	For $h \neq 1/2$, some difficulties arise since $s_{\xi_B}(t)$ is an oscillating sum with long range cancellation . To circumvent this issue, we will use Poisson summation formulation, allowing us to rewrite the difference $s_{\xi_B}(t)-\bar{s}_{\xi_B}(t)$ in a convenient way.
	
	\begin{proposition}\label{prop:poisson_summation_hyperuniform}
		There exists a positive finite Borel measure $\mu$ supported on $\R_{+}$ such that 
		\begin{equation*}
			s_{\xi_B}(t)-\bar{s}_{\xi_B}(t) = 2 \int_{0}^{+\infty}\sum_{n \neq 0}\frac{t^{1/h}u}{t^{2/h}u+(2\pi n+t)^{2}}\, \mu(\mathrm{d}u), \quad t> 0.
		\end{equation*}
		Moreover, $\mu$ satisfies
		\begin{equation*}
			e^{-t^{2h}} = \int_{0}^{+\infty}e^{-st} \, \mu(\mathrm{d}u), \quad t\geq 0.
		\end{equation*}
	\end{proposition}
	
	\begin{proof}
		Note that
		\begin{equation*}
			s_{\xi_B}(t) = \sum_{n\in \Z}\hat{f_{t}}(n)e^{itn}
		\end{equation*}
		with $\hat{f_{t}}(x)=e^{-\frac{1}{2}t^{2}|x|^{2h}}$. The application of Poisson summation formula and Fourier inverse theorem ensures that
		\begin{equation*}
			s_{\xi_B}(t)=2\pi\sum_{n \in \mathbb{Z}}f_{t}(2\pi n +t)
		\end{equation*}
		where 
		\begin{equation}\label{eq:poisson_formula}
			\begin{aligned}
				f_{t}(s) &= \frac{1}{2\pi}\int_{\R}e^{is\xi}e^{-\frac{1}{2}t^{2}|\xi|^{2h}} \, \mathrm{d}\xi\\
				&=\frac{1}{2\pi t^{1/h}}\int_{\R}e^{ist^{-1/h}\xi}e^{-\frac{1}{2}|\xi|^{2h}} \, \mathrm{d}\xi\\
				&=\frac{1}{\pi t^{1/h}}\Re\left(\int_{\R}e^{-ist^{-1/h}\xi}e^{-\frac{1}{2}\xi^{2}}\, \mathrm{d}\xi\right)
			\end{aligned}
		\end{equation}
		
		Now, for $h<1/2$, the function $\theta : t \to e^{-\frac{1}{2}t^{2h}}$ is right continuous, infinitely differentiable on $(0,+\infty)$ and completely monotone on $\R_{+}^{*}$ in the sense that 
		\begin{equation*}
			(-1)^{n}\theta^{(n)}(t)>0, \quad t>0, n \geq \N.
		\end{equation*}
		
		By Bernstein theorem \cite[Theorem 1.4]{Schilling2009}, it is then the Laplace transform of a positive finite measure $\mu$. Plugging this observation in \eqref{eq:poisson_formula} yields
		\begin{align*}
			f_{t}(s)&=\frac{1}{\pi t^{1/h}}\Re\left(\int_{0}^{+\infty}\int_{0}^{+\infty}e^{-(ist^{-1/h}+u)\xi} \, \mathrm{d}\xi \, \mu(\mathrm{d}u)\right)\\
			&=\frac{1}{\pi}\int_{0}^{+\infty}\frac{t^{1/h}u}{u^{2}t^{2/h}+s^{2}} \, \mu(\mathrm{d}u),
		\end{align*}
		which finally implies that
		\begin{equation*}
			s_{\xi_B}(t) = 2\int_{0}^{+\infty}\sum_{n \in \mathbb{Z}} \frac{t^{1/h}u}{u^{2}t^{2/h}+(2\pi n+t)^{2}} \,\mu(\mathrm{d}u).
		\end{equation*}
		
		A similar reasoning also gives
		\begin{equation*}
			\bar{s}_{\xi_B}(t)=2\int_{0}^{+\infty}\frac{t^{1/h}u}{u^{2}t^{2/h}+t^{2}} \, \mu(\mathrm{d}u).
		\end{equation*}
	\end{proof}
	
	An almost immediate consequence of Proposition \ref{prop:poisson_summation_hyperuniform} is the following bound on the difference between $s_{\xi_B}(t)$ and $\bar{s}_{\xi_B}(t)$. 
	
	\begin{lemma}
		If $0 <t <\pi$, then 
		\begin{equation*}
			s_{\xi_B}(t)-\bar{s}_{\xi_B}(t) \leq 2\Lambda(t^{1/h})
		\end{equation*}
		where 
		\begin{equation*}
			\Lambda(t) = \int_{0}^{+\infty} \coth(tu)-\frac{1}{tu} \, \mu(\mathrm{d}u)
		\end{equation*}
	\end{lemma}

	\begin{proof}
		Notice that 
		\begin{equation*}
			\frac{1}{t^{2/u}u^{2}+(2\pi n+t)^{2}} \leq \frac{1}{t^{2/h}u^{2}+\pi^{2}n^{2}}
		\end{equation*}
		as long as $|t|\leq \pi$ and $n\neq 0$. The conclusion of the proofs follows then from the fact that 
		\begin{equation*}
			\sum_{n \neq 0}\frac{1}{x^{2}+\pi^{2} n^{2}} = \coth(x)-\frac{1}{x}.
		\end{equation*}
	\end{proof}
	
	A straightforward computation shows that as $x \to 0$, 
	\begin{equation*}
		\coth(x)-\frac{1}{x} = \frac{x}{3}+o(x)
	\end{equation*}
	and one might expect by the dominated convergence theorem that 
	\begin{equation*}
		\Lambda(t) \sim \frac{t}{3}\int_{0}^{+\infty}u \, \mu(\mathrm{d}u),
	\end{equation*}
	so that
	\begin{equation*}
		s_{\xi_B}(t)-\bar{s}_{\xi_B}(t)=O(t^{1/h}).
	\end{equation*}
	Unfortunately, $\mu$ does not have a moment of order $1$ so the preceding computation fails. Yet, it is still possible to circumvent the issue and get: 
	
	\begin{lemma}\label{lem:lambda_asympotics_hyperuniform}
		As $t\to 0$,
		\begin{equation*}
			\Lambda(t) = O(t^{2h}).
		\end{equation*}
	\end{lemma}
	
	This latter lemma implies immediately Proposition \ref{prop:degree_hyperuniformity_palm_perturbed_lattice}. To prove the preceding result, the following estimates on the decay of $\mu$ will be needed.
	
	\begin{lemma}\label{lem:mu_estimate}
		\begin{equation*}
			\int_{0}^{t^{-1}} u \, \mu(\mathrm{d}u) \underset{t\to 0}{=} O\left(t^{2h-1}\right)
		\end{equation*}
		and 
		\begin{equation*}
			\int_{t^{-1}}^{+\infty} \mu(\mathrm{d}u)\underset{t\to 0}{=}O(t^{2h}).
		\end{equation*}
	\end{lemma}
	
	The proof of Lemma \ref{lem:mu_estimate} is delayed till the end of the proof of Lemma \ref{lem:lambda_asympotics_hyperuniform}.
	
	\begin{proof}[Proof of Lemma \ref{lem:lambda_asympotics_hyperuniform}]
		Write $\Lambda(t) = \Lambda_{1}(t)+\Lambda_{2}(t)$ where $\Lambda_{1}(t)$ is the contribution of the integral defining $\Lambda(t)$ on $[0,1/t]$.\medbreak
		
		To bound $\Lambda_{1}(t)$, notice that 
		\begin{equation*}
			\coth(x)-\frac{1}{x} \leq \frac{x}{3}, \quad x>0,
		\end{equation*}
		so that
		\begin{equation*}
			\Lambda_{1}(t) \leq \frac{t}{3} \int_{0}^{t^{-1}} u\, \mu(\mathrm{d}u)
		\end{equation*}
		and the latter is $O(t^{2h})$ according to Lemma \ref{lem:mu_estimate}.\medbreak
		
		As for $\Lambda_{2}(t)$, it is trivially bounded by
		\begin{equation*}
			\int_{t^{-1}}^{+\infty} \, \mu(\mathrm{d}u).
		\end{equation*}
		and the latter is $O(t^{2h})$ according to Lemma \ref{lem:mu_estimate}.
	\end{proof}
	
	\begin{proof}[Proof of Lemma \ref{lem:mu_estimate}]
		For the first inequality, differentiation under the integral sign gives 
		\begin{equation*}
			\int_{0}^{+\infty}e^{-tu}u \, \mu(\mathrm{d}u)= 2h t^{2h-1}e^{-t^{2h}}.
		\end{equation*}
		Hence,
		\begin{align*}
			\int_{0}^{t^{-1}} u \, \mu(\mathrm{d}u) &\leq e\int_{0}^{t^{-1}} ue^{-tu}\, \mu(\mathrm{d}u)\\
			&\leq e t^{2h-1}e^{-t^{2h}}.
		\end{align*}
		
		For the second inequality, notice that 
		\begin{align*}
			1-e^{-t^{2h}}&=\int_{0}^{+\infty}1-e^{-tu} \, \mu(\mathrm{d}u)\\
			&\geq \int_{t^{-1}}^{+\infty}1-e^{-tu}\,\mu(\mathrm{d}u)\\
			&\geq (1-e^{-1})\int_{t^{-1}}^{+\infty}\, \mu(\mathrm{d}u).
		\end{align*}
		On the other hand, 
		\begin{equation*}
			1-e^{-t^{2h}} = t^{2h}+o(t^{2h}).
		\end{equation*}
		This finishes the proof in the hyperuniform case.
	\end{proof}
	
	\subsection{Hyperfluctuating case $h>1/2$}
	
	In the hyperfluctuating case, a similar reasoning can be conducted with some tweaks	as for $h>1/2$, $t \to e^{-\frac{1}{2}t^{2h}}$ is no longer a completely monotone function. This issue can be circumvented with the following observation.
	
	\begin{proposition}\label{prop:poisson_summation_hyperfluctuating}
		There exists a finite positive measure supported on $\R^{+}$ so that
		\begin{equation*}
			s_{\xi_B}(t)-\bar{s}_{\xi_B}(t) = \frac{\sqrt{2\pi}}{t^{1/h}}\int_{0}^{+\infty}\sum_{n \neq 0}\frac{e^{-\frac{1}{2u}(2\pi n+t)^{2}t^{-2/h}}}{\sqrt{u}} \, \mu(\mathrm{d}u).
		\end{equation*}
		
		Moreover, $\mu$ satisfies
		\begin{equation*}
			e^{-\frac{1}{2}|t|^{2h}} = \int_{0}^{+\infty}e^{-\frac{1}{2}t^{2}u} \,\mu(\mathrm{d}u), \quad t \in \R.  
		\end{equation*}
	\end{proposition}
	
	\begin{proof}
		By Bernstein theorem, there exists a finite measure $\mu$ such that
		\begin{equation*}
			e^{\frac{1}{2}\xi^{h}} = \int_{0}^{+\infty}e^{-\frac{1}{2}\xi u} \, \mu(\mathrm{d}u), \quad \xi>0.
		\end{equation*}
		
		Hence,
		\begin{equation*}
			e^{\frac{1}{2}|\xi|^{2h}} = \int_{0}^{+\infty}e^{-\frac{1}{2}\xi^{2} u} \, \mu(\mathrm{d}u).
		\end{equation*}
		Now, following the same lines as in the proof of Proposition \ref{prop:poisson_summation_hyperuniform}, one derives the equations
		\begin{equation*}
			s_{\xi_B}(t)=\frac{\sqrt{2\pi}}{t^{1/h}} \int_{0}^{+\infty}\sum_{n \in \mathbb{Z}}\frac{e^{-\frac{1}{2u}(t+2\pi n)^{2}t^{-2/h}}}{\sqrt{u}} \, 	\mu(\mathrm{d}u)
		\end{equation*}
		and
		\begin{equation*}
			\bar{s}_{\xi_B}(t) = \frac{\sqrt{2\pi}}{t^{1/h}}\int_{0}^{+\infty}\frac{e^{-\frac{1}{2u}t^{2-2/h}}}{\sqrt{u}} \, \mu(\mathrm{d}u).
		\end{equation*}
	\end{proof}
	
	The preceding result allows to bound the difference between $s_{\xi_B}$ and $\bar{s}_{\xi_B}$ in the following manner.
	\begin{corollary}
		If $0 < t < 2\pi$, then
		\begin{equation*}
			0 \leq s_{\xi_B}(t)-\bar{s}_{\xi_B}(t) \leq 2 \left(\sum_{n \in \mathbb{Z}}e^{-\frac{2^{2h}}{2}t^{2}|n|^{2h}}-\int_{\R} e^{-\frac{2^{2h}}{2}t^{2}|x|^{2h}} \mathrm{d}x\right)
		\end{equation*}
	\end{corollary}
	
	\begin{proof}
		Note that 
		\begin{equation*}
			e^{-\frac{1}{2u}(2\pi n+t)^{2}t^{-2/h}}\leq e^{-\frac{1}{8u}(2\pi n)^{2}t^{-2/h}}
		\end{equation*}
		if $0<t<\pi$, $t>0$ and $n\neq 0$. Thus,
		\begin{equation*}
			s_{\xi_B}(t)-\bar{s}_{\xi_B}(t) \leq \int_{0}^{+\infty} \sum_{n \neq 0}e^{-\frac{1}{8u}(2\pi n)^{2}t^{-2/h}}\, \frac{\mu(\mathrm{d}u)}{\sqrt{u}}.
		\end{equation*}
		
		By Poisson summation formula, 
		\begin{equation*}
			\frac{\sqrt{2\pi}}{t^{1/h}\sqrt{u}}\sum_{n \in \mathbb{Z}} e^{-\frac{1}{8}(2\pi n)^{2}t^{-2/h}} = 2 \sum_{n \in \mathbb{Z}}e^{-2n^{2}t^{2/h}u}
		\end{equation*}
		Integrating with respect to $\mu$, the previous sum is equals to
		\begin{equation*}
			2\sum_{n \in \mathbb{Z}}e^{-\frac{2^{2h}}{2}t^{2}|n|^{2h}}.
		\end{equation*}
		Finally, 
		\begin{equation*}
			\begin{aligned}
				\frac{\sqrt{2\pi}}{t^{1/h}}\int_{0}^{+\infty}\frac{1}{\sqrt{u}} \, \mu(\mathrm{d}u)&=2\int_{\R} \int_{0}^{+\infty}e^{-2 t^{2/h}x^{2}u} \,\mu(\mathrm{d}u)\, \mathrm{d}x\\
				&=2\int_{\R} e^{-\frac{2^{2h}}{2} t^{2}|x|^{2h}} \, \mathrm{d}x.
			\end{aligned}
		\end{equation*}
		Combining the two preceding equalities yields the desired result.
	\end{proof}

	\begin{proposition}\label{prop:estimate_hyperfluctuating_case}
		\begin{equation*}
			s_{\xi_B}(t)-\bar{s}_{\xi_B}(t) \underset{t\to 0}{=}O(1).
		\end{equation*}
	\end{proposition}
	
	\begin{proof}[Proof Proposition \ref{prop:estimate_hyperfluctuating_case}]
		From the previous inequality,
		\begin{align*}
			s_{\xi_B}(t)-\bar{s}_{\xi_B}(t) &\leq 2+4\sum_{n =1}^{+\infty}\int_{n-1}^{n}e^{-\frac{2^{2h}}{2}t^{2}|n|^{2h}}-e^{-\frac{2^{2h}}{2}t^{2}|x|^{2h}}\, \mathrm{d}x\\
			&\leq 2+4\sum_{n =1}^{+\infty}\int_{n-1}^{n}e^{-\frac{2^{2h}}{2}t^{2}|n-1|^{2h}}-e^{-\frac{2^{2h}}{2}t^{2}|n|^{2h}} \mathrm{d}x\\
			&=6.
		\end{align*}
	\end{proof}
	
	Proposition \ref{prop:estimate_hyperfluctuating_case} finishes the proof of Proposition \ref{prop:degree_hyperuniformity_palm_perturbed_lattice} in the hyperfluctuating case.
	
	\clearpage
	
	\appendix
	
	\section{Ergodicity of Gaussian process with stationary increments}\label{ap:Gaussian_proces_stationary_increments}
	
	The goal of this section is to prove Theorem \ref{thm:Maruyama_stationary_increments}, characterizing the ergodic and mixing properties of a $\R^{d}$-valued Gaussian $(B_{t})_{t\in \R^{d}}$ process with stationary increments in terms of its Levy measure $\mu$.\medbreak
	
	The proof is divided into two steps:\medbreak
	
	\textbf{Step 1}: We characterize the ergodicity of a real-valued Gaussian process $(B_{t})_{t\in \R^{d}}$ with stationary increments.\medbreak
	
	\textbf{Step 2}: Denoting by $B=(B^{(1)}, \dots,B^{(d)})$ the coordinates of $B$, we relate the ergodicity of $B^{(i)}$ to the ergodicity of $B$.
	
	\subsection{Real-valued Gaussian process with stationary increments}
	
	In this step, we let $B$ be a real-valued stochastic Gaussian process with stationary increments and continuous covariance. We denote by
	\begin{itemize}
		\item $\Sigma_{s,t}=\cov{B_{s}}{B_{t}}$ its covariance function, here $s,t\in \R^{d}$,
		\item $v(t)=\var{B_{t}}$ its variogram,
		\item $\mu$ its Levy measure.
	\end{itemize}
	
	The goal of this paragraph is to prove the following result.
	
	\begin{theorem}\label{thm:Maruyama_stationary_increments_1}
		Let $B$ be a real-valued Gaussian process with stationary increments.
		\begin{itemize}
			\item $B$ is ergodic (and weakly-mixing) if and only if $\mu$ is non-atomic,
			\item $B$ is mixing if and only if
			\begin{equation*}
				V_{a,b}(t) \underset{\|t\|\to +\infty}{\longrightarrow}  0, \quad a,b \in \R^{d},
			\end{equation*}
			where 
			\begin{align*}
				V_{a,b}(t)&=\cov{B_{-a}}{ B_{t+b}-B_{t}}\\
				&=v(t+a)+v(t+b)-v(t+a+b)-v(t).
			\end{align*}
		\end{itemize}
		In particular, if $\mu$ is absolutely continuous, then $B$ is mixing.
	\end{theorem}

	We recall some important properties of $\mu$. Namely it is defined as the unique symmetric Borel measure satisfying 
	\begin{equation}\label{eq:levy_measure}
		\int_{\R^{d}} (1\wedge \|x\|^{2}) \, \mu(\mathrm{d}x)<+\infty.
	\end{equation}
	and for which 
	\begin{equation}\label{eq:Levy_measure_variogram}
		v(t)=\int_{\R^{d}}|1-e^{i(x,t)}|^{2}\, \mu(\mathrm{d}x)
	\end{equation}
	
	Playing with the symmetry and equation \eqref{eq:levy_khintchine} yields
	\begin{equation}\label{eq:levy_measure_covariance}
		\Sigma_{s,t} = \int_{\R^{d}}(1-e^{i(x,s)})(1-e^{-i(x,t)})\, \mu(\mathrm{d}x).
	\end{equation}
	
	\begin{remark}
		The proof of Theorem \ref{thm:Maruyama_stationary_increments_1} is inspired by the proof of Maruyama's theorem \cite{Maruyama1949}. The latter characterizes the ergodicity of a real-valued stationary Gaussian process $(X_{t})_{t\in \R^{d}}$. Say  $X$ is stationary if the distribution of $(X_{t-s})_{t\in \R^{d}}$ does not depend on $s$. The law of a stationary Gaussian process is identified by its covariance function
		\begin{equation*}
			r(t)=\cov{X_{0}}{X_{t}}.
		\end{equation*}
		The latter is a real, positive definite function, and provided it is continuous, it has a spectral decomposition
		\begin{equation*}
			r(t)=\Re\left(\int_{\R^{d}}e^{i(t,x)}\, \lambda(\mathrm{d}x)\right)
		\end{equation*}
		where $\lambda$ is a positive Borel measure, known as the spectral measure. This is known as Bochner's theorem and we refer the reader to \cite[Theorem 4.11]{Schilling2009} for a proof.\medbreak
		
		The ergodic properties of $X$ are related to the spectral measure via Maruyama's theorem which is recalled hereafter.
		
		\begin{theorem}[Maruyama \cite{Maruyama1949}]
			Let $X$ be a real-valued stationary Gaussian process, then
			\begin{itemize}
				\item $X$ is ergodic (and weakly-mixing) if and only if $\lambda$ is non-atomic.
				\item $X$ is mixing if and only if 
				\begin{equation*}
					r(t)\underset{\|t\|\to +\infty}{\longrightarrow}0.
				\end{equation*}
			\end{itemize}
			In particular, if $\lambda$ is absolutely continuous, then $X$ is mixing.
		\end{theorem}
		
		Note the resemblance with Theorem \ref{thm:Maruyama_stationary_increments_1}. In fact, one can draw an analogy and consider that the Levy $\mu$ is to a Gaussian process with stationary increments what the spectral measure $\lambda$ is to a stationary Gaussian process. Similarly, the variogram $v$ of Gaussian process with stationary increments plays a similar role as the covariance function $r$ of a stationary Gaussian process.\medbreak
		
		This analogy will be useful as one can adapt the original proof of Maruyama to the framework of Gaussian processes with stationary increments. This is what we do hereafter.  
	\end{remark}

	We start with a classical lemma translating ergodicity and mixing conditions into simpler conditions. 
	
	\begin{lemma}\label{lem:ergodic_characterization}
		$B$ is 
		\begin{itemize}
			\item ergodic if
			\begin{equation*}
				\lim_{r\to +\infty}\frac{1}{\Vol{\mathsf{B}_{r}}}\int_{\mathsf{B}_{r}}\esp{f(B)\overline{g(\theta_{t}B)}}\, \mathrm{d}t=\esp{f(B)}\overline{\esp{g(B)}},
			\end{equation*}
			\item weakly-mixing if
			\begin{equation*}
				\lim_{r\to +\infty}\frac{1}{\Vol{\mathsf{B}_{r}}}\int_{\mathsf{B}_{r}}\left|\,\esp{f(B)\overline{g(\theta_{t}B)}}-\esp{f(B)}\overline{\esp{g(B)}}\,\right|^{2}\, \mathrm{d}s=0,
			\end{equation*}
			\item strong-mixing if
			\begin{equation*}
				\lim_{\|t\|\to +\infty}\esp{f(B)\overline{g(\theta_{t}B)}}=\esp{f(B)}\overline{\esp{g(B)}}
			\end{equation*}
		\end{itemize}
		for any $f,g$ of the forms 
		\begin{equation}\label{eq:exp_function}
			f(B)=\exp{i\sum_{j=1}^{n}a_{j}B_{t_{j}}}, \quad g(B)=\exp{i\sum_{j=1}^{n}b_{j}B_{t_{j}}}
		\end{equation} 
		with $a_{j}, b_{j} \in \R$ and $ t_{j} \in \R^{d}$.
	\end{lemma}
	
	\begin{proof}
		$B$ is ergodic (\textit{resp.} weakly-mixing or mixing) if and only if the relations in Lemma \ref{lem:ergodic_characterization} holds for any $f(B), g(B) \in \L^{2}(\P)$; see \cite[Section 2.7]{Einsiedler2011}. The results follows as any $f(B), g(B)$ can be approximated by a linear combination of functions of the form
		\begin{equation*}
			h(B)=\exp{i\sum_{j=1}^{n}a_{j}B_{t_{j}}}.
		\end{equation*}
	\end{proof}
	
	Now, as $(B_{t_{1}}, \dots, B_{t_{n}})$ is a Gaussian vector, so is $\sum_{i=1}^{n}a_{i}B_{t_{i}}$. The inner integrands in Lemma \ref{lem:ergodic_characterization} can thus be explicitly computed. This is the content of the following lemma whose proof is left to the reader.
	\begin{lemma}\label{lem:gaussian_law}
		Let $f$ and $g$ defined as in \eqref{eq:exp_function}. Then,
		\begin{align*}
			&\esp{f(B)} =\exp{-\frac{\Sigma_{f}^{2}}{2}},\\
			&\esp{\overline{g(B)}} =\exp{-\frac{\Sigma_{g}^{2}}{2}},\\
			&\esp{f(B)\overline{g(\theta_{t}B)}}=\exp{-\frac{\Sigma_{f}^{2}+\Sigma_{g}^{2}-2\Sigma_{f,g}(t)}{2}},\\
			\intertext{with}
			&\Sigma_{f}^{2}=\sum_{1 \leq j,k \leq n}a_{j}a_{k}\Sigma_{t_{j}, t_{k}},\\
			&\Sigma_{g}^{2}=\sum_{1 \leq j,k \leq n}b_{j}b_{k}\Sigma_{t_{j}, t_{k}},\\
			&\Sigma_{f,g}(t)=\sum_{1 \leq j,k \leq n}a_{j}b_{k}(\Sigma_{t_{j},t+t_{k}}-\Sigma_{t_{j},t})=\sum_{1 \leq j,k \leq n}a_{j}b_{k}V_{-t_{j},t_{k}}(t).
		\end{align*}
	\end{lemma}
	The two preceding lemmas makes it possible to prove the characterization of mixing Gaussian process with stationary increments.  
	
	\begin{proof}[Proof of Theorem \ref{thm:Maruyama_stationary_increments} (mixing)]
		With Lemma \ref{lem:ergodic_characterization} and Lemma \ref{lem:gaussian_law}, $B$ is mixing if and only if
		\begin{equation*}
			\lim_{\|t\|\to +\infty}\Sigma_{f,g}(t) = 0.
		\end{equation*}
		where $\Sigma_{f,g}$ is defined as in Lemma \ref{lem:gaussian_law}. Equivalently, this holds if and only if
		\begin{equation*}
			\lim_{\|t\|\to +\infty}V_{a,b}(t)=0, \quad a,b\in \R^{d}.
		\end{equation*}
	\end{proof}
	
	As for the characterization of ergodic and weakly-mixing weakly stationary Gaussian processes, we follow the ideas developed in \cite{Rosinski1997}. Note that equation \eqref{eq:levy_measure_covariance} yields
	\begin{equation*}
		V_{a,b}(t)=\int_{\R}e^{-i(t,x)}(1-e^{-i(x,a)})(1-e^{-i(x,b)})\, \mu(\mathrm{d}x).
	\end{equation*}
	Consequently,
	\begin{equation*}
		V_{a,b}=\widehat{\gamma}_{a,b}.
	\end{equation*}
	where
	\begin{equation*}
		\gamma_{a,b}(\mathrm{d}x) = (1-e^{-i(x,a)})(1-e^{-i(x,b)}) \, \mu(\mathrm{d}t)
	\end{equation*}
	is a finite signed measure according to \eqref{eq:levy_measure}. Here $\hat{.}$ refers to the standard Fourier transform for signed measures. Plugging the preceding expression in Lemma \ref{lem:gaussian_law} yields
	\begin{equation*}
		\Sigma_{f,g}=\widehat{\gamma}
	\end{equation*}
	where 
	\begin{equation}\label{eq:gamma}
		\gamma=\sum_{1 \leq j,k \leq n}a_{j}a_{k}\gamma_{t_{j},-t_{k}}.
	\end{equation}
	Now, the space of finite signed measures on $\R^{d}$ is a Banach algebra equipped with the convolution operator of measures $\gamma_{1}\gamma_{2}$ defined by
	\begin{equation*}
		\gamma_{1}\gamma_{2}(A)=\int_{\R^{d}\times \R^{d}}\mathds{1}(s+t \in A)\, \gamma_{1}(\mathrm{d}s)\,\gamma_{2}(\mathrm{d}t)
	\end{equation*} 
	and the norm of total variation. Hence, equation \ref{eq:gamma} yields 
	\begin{equation}\label{eq:fourier}
		\begin{aligned}
			\exp{\Sigma_{f,g}}&=\sum_{n=0}^{+\infty}\frac{\hat{\gamma}^{n}}{n!}\\
			&=\widehat{e^{\gamma}}.
		\end{aligned}
	\end{equation}
	
	This latter relation is fundamental and we will exploit it to prove the ergodic and weakly-mixing characterization of weakly-stationary Gaussian processes. This will require a last, but well-known lemma.
	
	\begin{lemma}[{{\cite[Theorem 3.2.3]{Lukacs1970}}}]\label{lem:Cesaro_convergence}
		Let $\nu$ be a signed measure in $\R^{d}$, then
		\begin{equation*}
			\lim_{\|t\|\to+\infty}\frac{1}{\Vol{\mathsf{B}_{r}}}\int_{\mathsf{B}_{r}}\widehat{\nu}(s) \, \mathrm{d}s = \nu(\{0\}).
		\end{equation*}
	\end{lemma}
	
	\begin{proof}[Proof of Theorem \ref{thm:Maruyama_stationary_increments} (ergodic)]
		With equation \eqref{eq:fourier}, Lemmas \ref{lem:ergodic_characterization} and \ref{lem:gaussian_law}, $B$ is ergodic if and only if 
		\begin{equation*}
			\lim_{r\to +\infty}\frac{1}{\Vol{\mathsf{B}_{r}}}\int_{\mathsf{B}_{r}}\widehat{e^{\gamma}}(t)\, \mathrm{d}t = 1
		\end{equation*}
		for any $\gamma$ defined as in \eqref{eq:gamma}.\medbreak
		
		With Lemma \ref{lem:Cesaro_convergence}, this condition is equivalent to
		\begin{equation}\label{eq:ergodic_condition}
			e^{\gamma}(\{0\})=1.
		\end{equation}
		Assume first that $\mu$ is non-atomic, then so is $\gamma^{n}$ for any $n \geq 1$. Hence, the only atom of
		\begin{equation*}
			e^{\gamma}=\sum_{n=0}^{+\infty}\frac{\gamma^{n}}{n!}
		\end{equation*}
		comes from the initial term $\gamma^{0}=\delta_{0}$. As such, $e^{\gamma}(\{0\})=1$ and $B$ is ergodic.\medbreak
		
		On the other hand, if $\mu$ has an atom, say at position $t$. As $\mu$ is symmetric, it has an atom at $-t$, whence $\mu^{2}(\{0\})\geq \mu(\{t\})^{2}$. If $\gamma(\mathrm{d}x)=|1-e^{iax}|^{2} \, \mu(\mathrm{d}x)$ with $e^{iat}\neq 1$, then $\gamma^{2}$ has an atom at $0$ and
		\begin{equation*}
			e^{\gamma}(\{0\}) \geq 1+\frac{\gamma^{2}(\{0\})}{2} > 1.
		\end{equation*}
		and $B$ is not ergodic.
	\end{proof}
	
	It remains to prove the equivalence between ergodicity and weakly-mixing in Theorem \ref{thm:Maruyama_stationary_increments}. This is done below.
	
	\begin{proof}[Proof of Theorem \ref{thm:Maruyama_stationary_increments} (weakly-mixing)]
		We will prove that $B$ is weakly-mixing if and only if it is ergodic. Assume $B$ is ergodic. Similarly to the ergodic case, $B$ is weakly-mixing if and only if 
		\begin{equation*}
			\lim_{r \to +\infty}\frac{1}{\Vol{\mathsf{B}_{r}}}\int_{\mathsf{B}_{r}}(1-e^{\widehat{\gamma}(t)})^{2}\, \mathrm{d}t = 0.
		\end{equation*}
		Expanding the inner integrand, this is equivalent to 
		\begin{equation*}
			\lim_{r\to +\infty}\frac{1}{\Vol{\mathsf{B}_{r}}}\int_{\mathsf{B}_{r}}1-2\widehat{e^{\gamma}}(t)+\widehat{e^{2\gamma}}(t)\, \mathrm{d}s = 0/
		\end{equation*}
		With Lemma \ref{lem:Cesaro_convergence}, the preceding condition translates into
		\begin{equation*}
			e^{2\gamma}(\{0\})-2e^{\gamma}(\{0\})+1=0.
		\end{equation*}
		As $B$ is ergodic, $\mu$ is non-atomic and mimicking the proof of the ergodic case,
		\begin{equation*}
			e^{2\gamma}(\{0\})-2e^{\gamma}(\{0\})+1=1-2+1=0.
		\end{equation*}
	\end{proof}
	
	We will finish the proof of Theorem \ref{thm:Maruyama_stationary_increments_1} if we show that the absolute continuity of $\mu$ implies that $B$ is mixing.
	
	\begin{proof}[Proof (absolute continuity implies mixing)]
		Fix $a,b \in \R^{d}$. With Levy-Kintchine formula \eqref{eq:levy_measure_covariance},
		\begin{equation*}
			V_{a,b}(t)=\int_{\R^{d}}e^{-i(t,x)} (1-e^{-i(a,x)})(1-e^{-i(b,x)})f(x) \, \mathrm{d}x
		\end{equation*}
		where $f$ is the density of $\mu$. In virtue of \eqref{eq:levy_measure}, $f$ satisfies
		\begin{equation*}
			\int_{\R^{d}}(1 \wedge \|x\|^{2})f(x) \, \mathrm{d}x<+\infty.
		\end{equation*}
		
		Note that $e^{i(x,t)} = 1+i(x,t)+o(\|x\|)$ in a neighborhood of $0$. Hence, the function
		\begin{equation*}
			\frac{(1-e^{-i(x,a)})(1-e^{-i(x,b)})}{\|x\|^{2}}
		\end{equation*}
		extends into a continuous function on $[-1,1]^{d}$, and as a consequence
		\begin{equation*}
			(1-e^{-i(a,x)})(1-e^{-i(b,x)})f(x)=g(x)\|x\|^{2}f(x)
		\end{equation*}
		is integrable and continuous on $[-1,1]^{d}$. By the Riemann-Lebesgue lemma,
		\begin{equation*}
			V_{a,b}^{(1)}(t)=\int_{[-1,1]^{d}}e^{-i(t,x)} (1-e^{-i(a,x)})(1-e^{-i(b,x)})f(x)\, \mathrm{d}x \underset{\|t\|\to 0}{\longrightarrow} 0.
		\end{equation*}
		Similarly, on $\R^{d}\backslash [-1,1]^{d}$, as
		\begin{equation*}
			(1-e^{-i(a,x)})(1-e^{-i(b,x)})
		\end{equation*} 
		is bounded. Since $f$ is integrable on $\R \backslash [-1,1]^{d}$, a similar reasoning yields
		\begin{equation*}
			V^{(2)}_{a,b}(t)=\int_{\R \backslash [-1,1]^{d}} e^{-i(x,t)}(1-e^{-i(a,x)})(1-e^{-i(b,x)}) f(x)\, \mathrm{d}s \underset{\|t\|\to +\infty}{\longrightarrow} 0.
		\end{equation*}
		All in all, 
		\begin{equation*}
			\lim_{\|t\|\to +\infty}V_{a,b}(t)=\lim_{\|t\|\to +\infty}V_{a,b}^{(1)}(t)+V_{a,b}^{(2)}(t)=0
		\end{equation*}
		and $B$ is mixing.
	\end{proof}
	
	\subsection{$\R^{d}$-valued Gaussian process with stationary increments}
	
	This section is devoted to the proof of Theorem \ref{thm:Maruyama_stationary_increments}, which generalizes Theorem \ref{thm:Maruyama_stationary_increments_1} for $\R^{d}$-valued Gaussian processes with stationary increments. Here, $\smash{B=(B^{(1)},\dots, B^{(d)})}$ refers to a $\R^{d}$-valued Gaussian process with stationary increments and independent coordinates. As before, we denote by
	\begin{itemize}
		\item $\Sigma_{s,t} = \cov{B_{s}}{B_{t}}$ the covariance function,
		\item $\Sigma_{t} =  \operatorname{diag}(v_{1}(t), \dots, v_{d}(t))$ the variogram of $B$, with $v_{i}$ variogram of $B^{(i)}$,
		\item $\mu=(\mu_{1}, \dots, \mu_{d})$ the family of Levy measures of $B$, with $\mu_{i}$ the Levy measure of $B^{(i)}$. 
	\end{itemize}
	
	We will prove the following result, which in combination of Theorem \ref{thm:Maruyama_stationary_increments_1}, will immediately imply Theorem \ref{thm:Maruyama_stationary_increments} .
	
	\begin{proposition}
		Let $B$ be a $\R^{d}$-valued Gaussian process with stationary increments. Then,
		\begin{itemize}
			\item $B$ is ergodic if and only if $B^{(i)}$ is ergodic, $1\leq i \leq d$,
			\item $B$ is weakly-mixing if and only if $B^{(i)}$ is weakly-mixing, $1\leq i \leq d$,
			\item  $B$ is mixing if and only if $B^{(i)}$ is mixing, $1\leq i \leq d$,
		\end{itemize}
	\end{proposition}
	
	\begin{proof}
		It is clear that if $B$ is ergodic, \textit{resp.} weakly mixing or mixing, then each coordinate $B^{(i)}$ is also ergodic, \textit{resp.} weakly mixing or mixing.\medbreak
		
		For the converse, we start with the mixing case. If each $B^{(i)}$ is mixing, then $B$ is mixing as product of independent mixing random variables.\medbreak
		
		If each $B^{(i)}$ is weakly-mixing, then with \cite[Corollary 2.37]{Einsiedler2011}, $B$ is weakly-mixing as product of weakly-mixing random variables.\medbreak
		
		Finally, if $B^{(i)}$ is ergodic, then it is also weakly-mixing according to Theorem \ref{thm:Maruyama_stationary_increments_1} and we conclude with the above case.
	\end{proof}
	
	This concludes our brief digression on ergodic Gaussian processes with stationary increments. Before closing this section, let us note that all of the preceding results are expected to extend to infinitely divisible processes with stationary increments, albeit at the cost of more involved arguments. In his seminal paper \cite{Maruyama1970}, Maruyama initiated the study of the ergodic and mixing properties of stationary infinitely divisible processes. Several of his arguments should be adaptable to the present setting of infinitely divisible processes with stationary increments. In particular, one may reasonably conjecture that any ergodic infinitely divisible process with stationary increments is weakly mixing, in analogy with the stationary case, see \cite{Rosinski1997}.

	\clearpage
	
	\bibliographystyle{alpha}
	\bibliography{Reference.bib}

\begin{thebibliography}{MVN68}

\bibitem[BH23]{Bjoerklund2023}
Michael Bj{\"o}rklund and Tobias Hartnick.
\newblock Hyperuniformity and non-hyperuniformity of quasicrystals.
\newblock {\em Mathematische Annalen}, 389(1):365--426, June 2023.

\bibitem[Ble86]{Bleistein1986}
Norman Bleistein.
\newblock {\em Asymptotic expansions of integrals}.
\newblock Dover books on mathematics. Dover Publications, New York, 1986.

\bibitem[Cos21]{Coste2021}
Simon Coste.
\newblock Order, fluctuations, rigidities.
\newblock Technical report, 2021.

\bibitem[Die04]{Diecker2004}
Antonius Diecker.
\newblock {\em Simulation of Fractional Brownian Motion}.
\newblock PhD thesis, Vrije Universiteit Amsterdam, 2004.

\bibitem[DN97]{Dietrich1997}
Claude Dietrich and Garry Newsam.
\newblock Fast and exact simulation of stationary gaussian processes through
  circulant embedding of the covariance matrix.
\newblock {\em SIAM Journal on Scientific Computing}, 18(4):1088--1107, July
  1997.

\bibitem[DVJ08]{Daley2008}
Daryl~J. Daley and David Vere-Jones.
\newblock {\em An Introduction to the Theory of Point Processes - General
  Theory and Structure}.
\newblock Probability and its Applications. Springer New York, 2nd edition,
  2008.

\bibitem[EW11]{Einsiedler2011}
Manfred Einsiedler and Thomas Ward.
\newblock {\em Ergodic Theory: with a view towards Number Theory}.
\newblock Graduate Texts in Mathematics. Springer London, 2011.

\bibitem[GK21]{Ghosh2021}
Subhroshekhar Ghosh and Manjunath Krishnapur.
\newblock Rigidity hierarchy in random point fields: Random polynomials and
  determinantal processes.
\newblock {\em Communications in Mathematical Physics}, 388(3):1205--1234,
  November 2021.

\bibitem[GL16]{Ghosh2016}
Subhro Ghosh and Joel Lebowitz.
\newblock Number rigidity in superhomogeneous random point fields.
\newblock {\em Journal of Statistical Physics}, 166(3–4):1016--1027, October
  2016.

\bibitem[GP17]{Ghosh2017}
Subhroshekhar Ghosh and Yuval Peres.
\newblock Rigidity and tolerance in point processes: Gaussian zeros and ginibre
  eigenvalues.
\newblock {\em Duke Mathematical Journal}, 166(10), July 2017.

\bibitem[HS13]{Holroyd2013}
Alexander Holroyd and Terry Soo.
\newblock Insertion and deletion tolerance of point processes.
\newblock {\em Electronic Journal of Probability}, 18, January 2013.

\bibitem[Kal17]{Kallenberg2017}
Olav Kallenberg.
\newblock {\em Random Measures, Theory and Applications}.
\newblock Probability Theory and Stochastic Modelling. Springer International
  Publishing, 2017.

\bibitem[LP17]{Last2017}
G{\"u}nter Last and Mathew Penrose.
\newblock {\em Lectures on the Poisson Process}.
\newblock Institute of Mathematical Statistics Textbooks. Cambridge University
  Press, October 2017.

\bibitem[LR25a]{LachiezeRey2025a}
Raphaël Lachièze-Rey.
\newblock Hyperuniform random measures, transport and rigidity, 2025.

\bibitem[LR25b]{LachiezeRey2025}
Raphaël Lachièze-Rey.
\newblock Rigidity of random stationary measures and applications to point
  processes, 2025.

\bibitem[Luk70]{Lukacs1970}
Eugene Lukacs.
\newblock {\em Characteristic functions}.
\newblock Griffin books of cognate interest. Griffin, London, 2nd edition,
  1970.

\bibitem[Mar49]{Maruyama1949}
Gisiro Maruyama.
\newblock The harmonic analysis of stationary stochastic process.
\newblock {\em Memoirs of the Faculty of Science, Kyushu University. Series A,
  Mathematics}, 4(1):45--106, July 1949.

\bibitem[Mar70]{Maruyama1970}
Gisiro Maruyama.
\newblock Infinitely divisible processes.
\newblock {\em Theory of Probability and Its Applications}, 15(1):1--22, 1970.

\bibitem[MVN68]{Mandelbrot1968}
Benoit~B. Mandelbrot and John~W. Van~Ness.
\newblock Fractional brownian motions, fractional noises and applications.
\newblock {\em SIAM Review}, 10(4):422--437, October 1968.

\bibitem[RZ97]{Rosinski1997}
Jan Rosinski and Tomasz Zak.
\newblock The equivalence of ergodicity and weak mixing for infinitely
  divisible processes.
\newblock {\em Journal of Theoretical Probability}, 10(1):73--86, 1997.

\bibitem[SSV09]{Schilling2009}
René~L. Schilling, Renming Song, and Zoran Vondraček.
\newblock {\em Bernstein Functions: Theory and Applications}.
\newblock Studies in Mathematics. Walter de Gruyter, 2nd edition, December
  2009.

\bibitem[Tho25]{Thomassey2025}
Loïc Thomassey.
\newblock Perturbed palm measures.
\newblock August 2025.

\bibitem[Tor18]{Torquato2018}
Salvatore Torquato.
\newblock Hyperuniform states of matter.
\newblock {\em Physics Reports}, 745:1--95, June 2018.

\bibitem[TS03]{Torquato2003}
Salvatore Torquato and Frank~H Stillinger.
\newblock Local density fluctuations, hyperuniformity, and order metrics.
\newblock {\em Physical Review E}, 68(4), October 2003.

\bibitem[Zol01]{Zolotarev2001}
Vladimir~M. Zolotarev.
\newblock {\em One-dimensional stable distributions}.
\newblock Number~65. American Mathematical Society, 2001.

\end{thebibliography}
	
\end{document}